 \documentclass[12pt,a4wide]{article}
 \usepackage[utf8]{inputenc}

\usepackage{graphicx}

\usepackage{amsfonts}
\usepackage{latexsym}
\usepackage{ amssymb, bm, amsmath, amsthm }
\usepackage{float}
\usepackage{multirow}
\usepackage{a4wide}
\usepackage{enumerate}
\usepackage{url}
\usepackage[dvipsnames]{xcolor}
\usepackage{todonotes}
\usepackage[noadjust]{cite}
\usepackage{soul}

\usepackage{array}
\usepackage{makecell}
\usepackage{subcaption}
\setcellgapes{5pt}

\definecolor{lightblue}{rgb}{0.68,0.85,0.9}

\usepackage{tikz}
\usetikzlibrary{shapes}
\usetikzlibrary{calc}

\pgfdeclarelayer{edgelayer}
\pgfdeclarelayer{nodelayer}
\pgfsetlayers{edgelayer,nodelayer,main}

\tikzstyle{Empty}=[fill=white, draw=black, shape=circle]
\tikzstyle{Vertex}=[fill=white, draw=black, shape=circle, inner sep=1pt, outer sep=0pt]
\tikzstyle{Gray}=[fill={rgb,255: red,128; green,128; blue,128}, draw=black, shape=circle]
\tikzstyle{Red}=[fill=Gray!60, draw=black, shape=circle]
\tikzstyle{Square}=[fill=white, draw=black, shape=rectangle]
\tikzstyle{Blue}=[fill=cyan, draw=black, shape=circle]
\tikzstyle{BlueSquare}=[fill=cyan, draw=black, shape=rectangle]
\tikzstyle{RedSquare}=[fill=red, draw=black, shape=rectangle]
\tikzstyle{Green}=[fill=green, draw=black, shape=circle]
\tikzstyle{GreenSquare}=[fill=green, draw=black, shape=rectangle]
\tikzstyle{GraySquare}=[fill={rgb,255: red,128; green,128; blue,128}, draw=black, shape=rectangle]

\tikzstyle{Normal}=[-, draw=black, fill=none]
\tikzstyle{Arrow}=[->]
\tikzstyle{Arrow2}=[<-]

\newtheorem{theorem}{Theorem}[section]

\newtheorem{lemma}[theorem]{Lemma}
\newtheorem{definition}[theorem]{Definition}

\newtheorem{remark}[theorem]{Remark}

\title{Optimal $(1,\le \ell)$-locating-dominating codes in infinite triangular grid\thanks{A shorter version of this paper has appeared in the conference proceedings of the 37th International Workshop on Combinatorial Algorithms (IWOCA 2026), \cite{das20261}.}}
\author{{\sc Soura Sena Das}$\,^{a}$, {\sc  Tuomo Lehtilä}$\,^{b}$, {\sc Sagnik Sen}$\,^{c}$ \\
\mbox{}\\
{\small $(a)$ Indian Statistical Institute, Kolkata, India}\\
{\small $(b)$ Department of Mathematics and Statistics, University of Turku, FI-20014, Turku, Finland}\\
{\small $(c)$ Indian Institute of Technology Dharwad, India}}

\begin{document}

\maketitle

\begin{abstract}
We study two variants of location-domination in infinite triangular grids which allow locating up to $\ell$ vertices simultaneously. A locating-dominating code $S$ is a dominating vertex set such that every vertex outside of it has a unique neighborhood within the code $S$. As this neighborhood is unique, we may use it as an identifier for locating the considered vertex outside of the vertex set. With the variants we study, we may also locate sets of vertices instead of only single vertices. In particular, we obtain the exact minimum densities of these types of sets in the infinite triangular grid for every value of $\ell$.   
%
%

\end{abstract}


\textbf{Keywords:}
 Locating-dominating codes, $(r,\leq \ell)$-LDA codes, $(r,\leq \ell)$-LDB codes, triangular grid, density of codes.

\section{Introduction}
In this paper, we consider infinite undirected graphs. A vertex set $S$ is \textit{dominating} if every vertex is in $S$ or adjacent to a vertex in $S$. A dominating set $S$ is \textit{locating-dominating} in graph $G$ if every vertex in $V(G)\setminus S$ has a unique neighborhood in $S$ (that is, $S$ is a dominating set with an additional separation property). Locating-dominating codes (or sets) were introduced by Rall and Slater in the 1980s~\cite{rall1984location, slater1987domination, slater1988dominating} and since then they and their variants have been studied in a wide number of articles, see an online bibliography containing over 500 articles on the topic in~\cite{LDbiblio}.

A major branch of research on the topic has been studying the optimal (smallest) density they can obtain in infinite graphs such as triangular, king, square or hexagonal grids (see for example~\cite{cohen2000bounds, honkala2006locating, jean2025optimal}). One may find over 40 such articles considering grids in the online bibliography~\cite{LDbiblio}. While the optimal density is known for many combinations of grids and variants of location-domination, some have remained open or proven hard to solve, such as identifying codes in the hexagonal grid~\cite{cohen2000bounds, cukierman2013new, salo2024finding}. In this paper, we give the exact optimal cardinalities for $(1,\leq \ell)$-locating-dominating codes of types $A$ and $B$ ($(1,\leq \ell)$-LDA and -LDB for short) for each value of $\ell$ in the infinite triangular grid. Previously only the value with $\ell=1$ was known~\cite{honkala2006optimal}. In the cases of other grids, the optimal densities for $(1,\leq \ell)$-LDA and -LDB codes are known for 
king grids~\cite{pelto2011, pelto2012r, pelto2013optimal} and for square grids (except for the case of $(1,\leq 2)$-LDB)~\cite{das2026density}.

The early motivation for studying location-domination stemmed from locating in structures described by graphs. In~\cite{slater1987domination} for \textit{safeguard analysis of critical facilities} and in~\cite{kar-chak-lev}, the authors introduced the concept of identifying codes (which can be understood as a special type of locating-dominating code) for the purpose of locating malfunctioning processors in multiprocessor systems. One interesting structure for the study from this perspective was the infinite grids (here, the infinite grid could be considered to approximate finite architectures based on the grid structures). More recently, these concepts have been connected to neighborhood complexity~\cite{bonnet2024neighbourhood} which is useful for kernelization algorithms.

A perceived weakness of both locating-dominating and identifying codes was their inability to simultaneously locate more than a single object, for example, in the case of multiple simultaneously malfunctioning processors or in sensor networks. For this purpose, $(1,\leq \ell)$-identifying codes were formally introduced 
in~\cite{laihonen2001codes} and briefly discussed in~\cite{kar-chak-lev}
to locate at most $\ell$ objects simultaneously. Later, in~\cite{honkala2004locating}, $(1,\leq \ell)$-locating-dominating codes were introduced. However, as there were two different natural definitions for these concepts, two types, type $A$ and type $B$, were considered. When $\ell=1$, the definitions give the usual locating-dominating codes. When $\ell\geq\Delta(G)+1$, where $\Delta(G)$ is the maximum degree of graph $G$, a $(1,\leq \ell)$-LDB code is equivalent to a complement of an open-irredundant set and also with a multiple intruder locating-dominating code (MILD code)~\cite{lehtilathree}. 

   \medskip
    
    \noindent\textbf{Preliminaries:}
    We follow West~\cite{D.B.West} for standard graph theoretic notation and terminology.
    The graphs considered in this article are undirected and do not have any loops or parallel edges, but they can be infinite.
    
    Given a graph $G$ and a vertex $v$, the ball with center $v$ and radius $r$ is given by 
    $$B_r(v) = \{u~|~d(u,v) \leq r\},$$ 
    where $d(u,v)$ denotes the distance between $u$ and $v$ in $G$. 
    For a vertex subset $F \subseteq V(G)$, the ball around $F$ with radius $r$ is given by 
    $$B_r(F) = \displaystyle{\bigcup_{v\in F}} B_r(v).$$
    A non-empty vertex subset $C \subseteq V(G)$ is called a \textit{code} and its vertices are called \textit{codewords}
    while the vertices of $V(G) \setminus C$ are called \textit{non-codewords}.
    For each vertex $u\in V(G)$,  the \emph{$(I,r)$-set} of $u$ is defined as 
  $$I_r(u) = B_r(u) \cap C.$$

   Furthermore, the \emph{$(I,r)$-set} of a vertex set $F$ is defined as 
    $$I_r(F) = B_r(F) \cap C.$$
    A generalization of locating-dominating codes was introduced by Honkala et al.~\cite{honkala2004locating}.
    In particular, $C$ is called an \textit{$(r, \leq \ell)$-locating-dominating code of type A}, or simply \textit{$(r, \leq \ell)$-LDA code} if $I_r(F_1) \neq I_r(F_2)$ whenever 
    $F_1 \neq F_2$, $F_1 \cap C = F_2 \cap C$, and $|F_1|, |F_2| \leq \ell$ 
    for all $F_1, F_2 \subseteq V(G)$.
    Similarly, $C$ is called an \textit{$(r, \leq \ell)$-locating-dominating code of type B}, or simply \textit{$(r, \leq \ell)$-LDB code} if $I_r(F_1) \neq I_r(F_2)$ whenever 
    $F_1 \neq F_2$, and $|F_1|, |F_2| \leq \ell$ for all $F_1, F_2 \subseteq V(G) \setminus C$.
    Since the objective of our article is to study the 
    $(1, \leq \ell)$-LDA  and the $(1, \leq \ell)$-LDB codes of infinite triangular grids, for simplicity, we will refer to
    the ($I, 1$)-sets as $I$-sets, and will use the notation $I(\cdot)$ instead of $I_1(\cdot)$.

      \medskip

    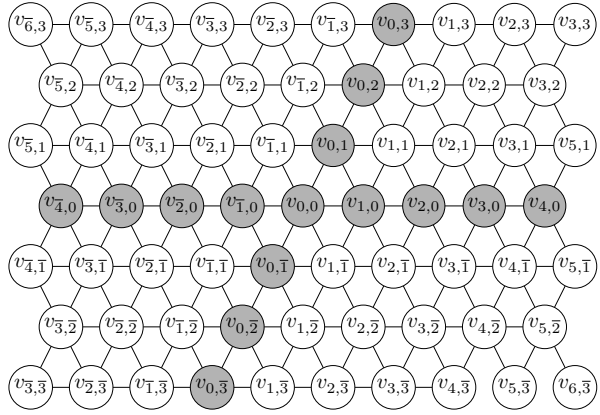
\begin{figure}[t]
        \centering
        \scalebox{0.8}{


\begin{tikzpicture}
    \begin{pgfonlayer}{nodelayer}

        \foreach \x in {-3,...,3} {
            \foreach \y in {-3,...,3} {

                \pgfmathsetmacro{\xoffset}{\y * 0.5}
            
                \pgfmathtruncatemacro{\absy}{abs(\y)}
                \pgfmathtruncatemacro{\mody}{mod(\absy,2)}
                \pgfmathsetmacro{\realx}{\x + \xoffset}

                \pgfmathsetmacro{\fillcolor}{(\x == 0 || \y==0) ? "gray!60" : "white"}

                \pgfmathtruncatemacro{\absx}{abs(\x)}
                \pgfmathtruncatemacro{\labx}{\x < 0 ? 1 : 0}
                \pgfmathtruncatemacro{\laby}{\y < 0 ? 1 : 0}

                \node [style=Vertex, fill=\fillcolor] (v\x\y) at (\realx, \y)
                {\footnotesize$
                    v_{
                    \ifnum\labx=1 \overline{\absx}\else\x\fi,
                    \ifnum\laby=1 \overline{\absy}\else\y\fi}$};
            }
        }
        
        \node [style=Vertex, fill=white] (v-4-1) at (-4.5, -1)          {\footnotesize$v_{\overline{4},\overline{1}}$};
        
        \node [style=Vertex, fill=gray!60] (v-40) at (-4, 0)          {\footnotesize$v_{\overline{4},0}$};
        
        \node [style=Vertex, fill=white] (v-41) at (-3.5, 1)          {\footnotesize$v_{\overline{4},1}$};
        \node [style=Vertex, fill=white] (v-51) at (-4.5, 1)          {\footnotesize$v_{\overline{5},1}$};
        
        \node [style=Vertex, fill=white] (v-42) at (-3, 2)          {\footnotesize$v_{\overline{4},2}$};
        \node [style=Vertex, fill=white] (v-52) at (-4, 2)          {\footnotesize$v_{\overline{5},2}$};
        
        \node [style=Vertex, fill=white] (v-43) at (-2.5, 3)          {\footnotesize$v_{\overline{4},3}$};
        \node [style=Vertex, fill=white] (v-53) at (-3.5, 3)          {\footnotesize$v_{\overline{5},3}$};
        \node [style=Vertex, fill=white] (v-63) at (-4.5, 3)          {\footnotesize$v_{\overline{6},3}$};

        \node [style=Vertex, fill=white] (v51) at (4.5, 1)          {\footnotesize$v_{5,1}$};

        \node [style=Vertex, fill=gray!60] (v40) at (4, 0)          {\footnotesize$v_{4,0}$};

        \node [style=Vertex, fill=white] (v4-1) at (3.5, -1)          {\footnotesize$v_{4,\overline{1}}$};
        \node [style=Vertex, fill=white] (v5-1) at (4.5, -1)          {\footnotesize$v_{5,\overline{1}}$};

        \node [style=Vertex, fill=white] (v4-2) at (3, -2)          {\footnotesize$v_{4,\overline{2}}$};
        \node [style=Vertex, fill=white] (v5-2) at (4, -2)          {\footnotesize$v_{5,\overline{2}}$};

        \node [style=Vertex, fill=white] (v4-3) at (2.5, -3)          {\footnotesize$v_{4,\overline{3}}$};
        \node [style=Vertex, fill=white] (v5-3) at (3.5, -3)          {\footnotesize$v_{5,\overline{3}}$};
        \node [style=Vertex, fill=white] (v6-3) at (4.5, -3)          {\footnotesize$v_{6,\overline{3}}$};

    \end{pgfonlayer}
    
    \begin{pgfonlayer}{edgelayer}
        \foreach \y in {-3,...,3} {
            \foreach \x in {-3,...,2} {
                \pgfmathtruncatemacro{\nextx}{\x + 1}
                \draw (v\x\y) to (v\nextx\y);
            }
        }
        
        \foreach \x in {-3,...,3} {
            \foreach \y in {-3,...,2} {
                \pgfmathtruncatemacro{\nexty}{\y + 1}
                \draw (v\x\y) to (v\x\nexty);
            }
        }
    
        \foreach \x in {-2,...,3} {
            \foreach \y in {-3,...,2} {
                \pgfmathtruncatemacro{\prevx}{\x - 1}
                \pgfmathtruncatemacro{\nexty}{\y + 1}
                \draw (v\x\y) to (v\prevx\nexty);
            }
        }

        \draw (v-43) to (v-33);
        \draw (v-43) to (v-32);
        \draw (v-43) to (v-42);
        \draw (v-43) to (v-53);

        \draw (v-53) to (v-42);
        \draw (v-53) to (v-52);
        \draw (v-53) to (v-63);

        \draw (v-63) to (v-52);

        \draw (v-42) to (v-32);
        \draw (v-42) to (v-31);
        \draw (v-42) to (v-41);
        \draw (v-42) to (v-52);

        \draw (v-52) to (v-51);
        \draw (v-52) to (v-41);

        \draw (v-41) to (v-30);
        \draw (v-41) to (v-31);
        \draw (v-41) to (v-40);

        \draw (v-51) to (v-40);
        \draw (v-51) to (v-41);

        \draw (v-40) to (v-30);
        \draw (v-40) to (v-3-1);
        \draw (v-40) to (v-4-1);

        \draw (v-4-1) to (v-3-2);
        \draw (v-4-1) to (v-3-1);

        \draw (v51) to (v32);
        \draw (v51) to (v31);
        \draw (v51) to (v40);

        \draw (v40) to (v30);
        \draw (v40) to (v31);
        \draw (v40) to (v4-1);
        \draw (v40) to (v5-1);

        \draw (v4-1) to (v30);
        \draw (v4-1) to (v3-1);
        \draw (v4-1) to (v5-1);
        \draw (v4-1) to (v4-2);
        \draw (v4-1) to (v5-2);

        \draw (v4-2) to (v3-2);
        \draw (v4-2) to (v3-1);
        \draw (v4-2) to (v5-3);
        \draw (v4-2) to (v4-3);
        \draw (v4-2) to (v5-2);

        \draw (v5-2) to (v6-3);
        \draw (v5-2) to (v5-3);
        \draw (v5-2) to (v5-1);

        \draw (v4-3) to (v3-2);
        \draw (v4-3) to (v3-3);

    \end{pgfonlayer}
\end{tikzpicture}
}
        \caption{The vertices of the infinite triangular grid are denoted by $v_{i,j}$ where $i,j \in \mathbb{Z}$. 
        Moreover, we use the notation $\overline{i}$ instead of $-i$ in the subscript in the index of $v$. 
        The infinite triangular grid is embedded in $\mathbb{R}^2$
        with the coordinates of the vertex $v_{i,j}$ being 
        $\left(  i + \frac{1}{2}\ j, \frac{\sqrt{3}}{2}\ j \right)$. The above figure depicts this particular embedding, which we will follow in this article as a convention, unless otherwise stated.}
        \label{fig:coordinate}
    \end{figure}

    When $\ell=1$, both $(1,\leq \ell)$-LDA and -LDB codes give an equivalent definition with the usual locating-dominating codes: a code $C$ is a $(1,\leq 1)$-LDA or -LDB code if and only if $C$ is dominating and $I(x)\neq I(y)$ for each $x\neq y$ where $x,y\in V(G)\setminus C$. Note that here the domination condition is due to allowing an empty set $F$ in the definitions of LDA and LDB codes. 
    Moreover, we know that any $(1,\leq \ell)$-LDA code is also a  $(1,\leq \ell)$-LDB code. Similarly, we have that any $(1,\leq \ell+1)$-LDA ($(1,\leq \ell+1)$-LDB) code is also a $(1,\leq \ell)$-LDA ($(1,\leq \ell)$-LDB) code~\cite{honkala2004locating}.

In this article, our focus is on $(1,\leq \ell)$-locating-dominating codes of types A and B on the infinite triangular grid, denoted by $\mathbf{G_6}$. Although we provide some results also for the case $r>1$.
    
   \begin{definition}
       Let $\mathbf{G_6}$ be the infinite triangular grid with vertex set $V(\mathbf{G_6})$ and edge set $E(\mathbf{G_6})$. 
    We define $V(\mathbf{G_6})$ as a collection of the following points of the plane $\mathbb{R}^2$ 
    $$ V(\mathbf{G_6}) =  \left\{ v_{i,j} = \left(  i + \frac{1}{2}\ j, \frac{\sqrt{3}}{2}\ j \right) \mid i,j\in \mathbb{Z}\right\}, $$ and two vertices $ v_{i,j},v_{i',j'} $ are adjacent iff their Euclidean distance is exactly one. 
    In other words, $ v_{i,j}$ is adjacent to 
    $v_{i',j'}$ in $\mathbf{G_6}$ if one of the following holds:
    \begin{enumerate}[(i)]
        \item $i = i'$, $|j-j'| = 1$,
        \item $j = j'$, $|i-i'| = 1$,
        \item $(i-i') = (j'-j) = 1$,
        \item $(i'-i) = (j-j') = 1$. 
    \end{enumerate}
For notational convenience, 
    the negative subscripts of the type $-i$ of the vertices of 
    $\mathbf{G_6}$ are replaced by $\overline{i}$. For instance, we write $v_{1,\overline{1}}$ instead of $v_{1, -1}$. 
    See Figure~\ref{fig:coordinate} for a reference illustration of a portion of the infinite triangular grid. 
   \end{definition}

    Our interest is in the relative size of vertex sets in the infinite graphs. To measure this, we introduce the following notion of density. 
    
  \begin{definition}\label{DefDens}
      The \textit{density} of a code $C$ of $\mathbf{G_6}$ is given by 
    $$D(C) = \limsup_{n \rightarrow \infty} \frac{|C \cap B_n(v_{0,0})|}{|B_n(v_{0,0})|}.$$
  \end{definition}

   \begin{remark}\label{rem density}
        We note that in~\cite{sampaio2024density}, it has been shown for connected graphs with the slow-growth property~\cite{sampaio2024density} (such as the infinite triangular grid), that the density of a vertex set can be defined using any center point (in our case, we choose the center point as $v_{0,0}$). While Definition~\ref{DefDens} is quite technical, often the simplest method for finding the density of a given periodic vertex set $C$ in an infinite grid, is to find a periodic tiling covering the entire grid in which each tile has a constant size and contains a constant number of vertices in $C$. For example, if each tile has $n$ vertices out of which $c$ belong to code $C$, then we have $D(C)=\frac{c}{n}$,~\cite[Lemma 3.4]{sampaio2024density}.
   \end{remark}

    An $(r, \leq \ell)$-LDA code (resp., $(r, \leq \ell)$-LDB code) is \textit{minimum} if there does not exist any 
    $(r, \leq \ell)$-LDA code (resp., $(r, \leq \ell)$-LDB code) with a strictly smaller density. Moreover, 
    the density of a minimum $(r, \leq \ell)$-LDA code (resp., $(r, \leq \ell)$-LDB code) for graph $G$ is denoted by $D^A_{r,\ell}(G)$ (resp., $D^B_{r,\ell}(G)$). In the case of the infinite triangular grid $\mathbf{G_6}$, we use the shorthand notation $D^A_{r,\ell}$ (resp., $D^B_{r,\ell}$) instead of $D^A_{r,\ell}(\mathbf{G_6})$ (resp., $D^B_{r,\ell}(\mathbf{G_6})$). 
    Our objective is to compute the $D^A_{1,\ell}$ and $D^B_{1,\ell}$  for all $\ell \geq 1$.

\medskip

    \noindent\textbf{Related research:} In the following, we present the related research. As the literature on location-domination is quite broad, we focus on related results in the infinite triangular grid and on results for $(r,\leq \ell)$-LDA/LDB codes in other types of infinite grids.

    Recall that when $\ell=1$, both $(1,\leq 1)$-LDA and -LDB codes are equivalent to the usual locating-dominating codes. In~\cite{honkala2006optimal}, Honkala has given the exact minimum density of locating-dominating codes in the infinite triangular grid.
    \begin{theorem}[\cite{honkala2006optimal}]\label{thmTri1LD}
        The density of a minimum locating-dominating code for the infinite triangular grid $\mathbf{G_6}$ is:
        $$D^A_{1,1}=D^B_{1,1}=\frac{13}{57}.$$
    \end{theorem}

    In~\cite{junnila2018regular}, the authors have introduced two new types of locating-dominating codes (self-locating-dominating, SLD, and solid-locating-dominating, DLD) which can locate a single vertex and in addition detect if the objective is to locate more than one vertex. Moreover, any $(1,\leq 2)$-LDA code is also an SLD code which is also an LD-code and any $(1,\leq 2)$-LDB code is also an DLD code which is also an LD-code. Hence, any construction for $(1,\leq 2)$-LDA/LDB codes gives an upper bound for the minimum density of SLD/DLD codes and any lower bound on the minimum density of SLD/DLD codes gives a lower bound for the minimum density. In~\cite{junnila2026new}, the authors have presented quite simple arguments which lead to the exact minimum density of $\frac{1}{4}$ for SLD/DLD codes in the infinite triangular grid. While this is the best known lower bound for $D^A_{1,2}$ and $D^B_{1,2}$, unfortunately similar simple arguments do not work for proving their minimum densities.

    In~\cite{honkala2004identification}, the authors have considered $(1,\leq 2)$-identifying codes in the infinite triangular grids. Moreover, any $(1,\leq 2)$-identifying code is also a $(1,\leq 2)$-LDA/LDB code and hence any upper bound for the minimum density of $(1,\leq 2)$-identifying codes is also an upper bound for the minimum density of $(1,\leq 2)$-LDA/LDB codes. In~\cite{honkala2004identification}, the authors have shown that the minimum density of $(1,\leq 2)$-identifying codes in the infinite triangular grids is $\frac{9}{16}$. Moreover, they have shown that there does not exist any $(1,\leq \ell)$-identifying codes for $\ell\geq3$ in the infinite triangular grid. Moreover, it is known that the density of optimal $(r,\leq 2)$-identifying codes in the triangular grid tends to 0 as $r$ increases (unlike in some other grids such as the square grid). As identifying codes provide upper bounds for locating-dominating codes, also the density of $(r,\leq 2)$-LDA/LDB tend to 0 as $r$ tends to infinity.

    As we have mentioned, previously $(1,\leq \ell)$-LDA and -LDB codes have been studied in the infinite king grid and the infinite square grid. For the infinite square grid $\mathbf{G_4}$ we know for LDA $D^A_{1,1}(\mathbf{G_4})=\frac{3}{10}$~(\cite{slater2002fault}), $D^A_{1,2}(\mathbf{G_4})=\frac{1}{2}$~(\cite{das2026density}), and $D^A_{r,\ell}(\mathbf{G_4})=1$ for each $r\geq1$ and $\ell\geq3$~(\cite{das2026density}). For LDB, we know $D^B_{1,1}(\mathbf{G_4})=\frac{3}{10}$~(\cite{slater2002fault}), $\frac{2}{5}\leq D^B_{1,2}(\mathbf{G_4})\leq\frac{5}{11}$~(\cite{das2026density}), $D^B_{1,\ell}(\mathbf{G_4})=\frac{1}{2}$~(\cite{das2026density}) for each $\ell\geq3$. In the case of the infinite king grid $\mathbf{G_8}$, we know $D^A_{1,1}(\mathbf{G_8})=D^B_{1,1}(\mathbf{G_8})=\frac{1}{5}$~(\cite{honkala2006locating}), $D^A_{1,2}(\mathbf{G_8})=D^B_{1,2}(\mathbf{G_8})=\frac{1}{3}$~(\cite{pelto2012r}), $D^A_{1,\ell}(\mathbf{G_8})=1$ for each $\ell\geq3$~(\cite{pelto2011}), $D^B_{1,3}(\mathbf{G_8})=\frac{3}{5}$~(\cite{pelto2013optimal}) and $D^B_{1,3}(\mathbf{G_8})=\frac{2}{3}$~(\cite{pelto2011}). To summarize, the only open case for radius $r=1$ in the king and square grids is obtaining $D^B_{1,2}(\mathbf{G_4})$.

    The larger values of $r$ have been studied through multiple articles in the case of the infinite king grid. However, the only known tight values are $D^B_{2,3}(\mathbf{G_8})=\frac{2}{3}$~(\cite{pelto2013optimal}), $D^B_{2,\ell}(\mathbf{G_8})=\frac{4}{5}$ for each $\ell\geq9$~(\cite{pelto2011}), $D^B_{r,3}(\mathbf{G_8})=\frac{r}{r+1}$ for each $r\geq3$~(\cite{pelto2013optimal}) and $D^B_{r,\ell}(\mathbf{G_8})=\frac{2r}{2r+1}$ for each $\ell>4r\geq12$~(\cite{pelto2011}). 
\medskip

    \noindent\textbf{Our contributions:}
    We summarize our findings in the following two theorems. 

    \begin{theorem}\label{thm lda}
        The density of a minimum $(r, \leq \ell)$-LDA code for the infinite triangular grid $\mathbf{G_6}$ for all $\ell \geq 1$ and some $r$ is given by:
        \begin{enumerate}[(a)]
            \item $D^A_{1,1} = \frac{13}{57}$~\cite{honkala2006optimal},
            \item $D^A_{1,2} = \frac{1}{3}$,
            \item $D^A_{r,\ell} = 1$, for all $\ell \geq 3$ and $r\geq1$. 
        \end{enumerate}
    \end{theorem}

    \begin{theorem}\label{thm ldb}
        The density of a minimum $(1, \leq \ell)$-LDB code for the infinite triangular grid $\mathbf{G_6}$ for all $\ell \geq 1$ is given by:
        \begin{enumerate}[(a)]
            \item $D^B_{1,1} = \frac{13}{57}$~\cite{honkala2006optimal},
            \item $D^B_{1,2} = \frac{1}{3}$,
            \item $D^B_{1,3} = \frac{4}{7}$,
            \item $D^B_{1,\ell} = \frac{2}{3}$, for all $\ell \geq 4$. 
        \end{enumerate}
    \end{theorem}

\medskip


 \noindent\textbf{Organization:}
    The rest of the article presents the proofs of Theorems~\ref{thm lda} and~\ref{thm ldb} in Section~\ref{sec proof}. We present upper bound constructions for Theorems~\ref{thm lda}(b) and~\ref{thm ldb}(b) in Section~\ref{secUpp12bb}, then in Section~\ref{sec1c} we prove Theorem~\ref{thm lda}(c). We continue in Section~\ref{sec2cupp} by giving a tight construction for Theorem~\ref{thm ldb}(c) and in Section~\ref{sec2dupp} for Theorem~\ref{thm ldb}(d). Section~\ref{subsection general technique} discusses our general proof technique for the remaining lower bounds. 
    Then, in the remaining sections we prove the lower bounds.
    We prove the lower bounds of Theorem~\ref{thm lda}(b) and Theorem~\ref{thm ldb}(b) in Section~\ref{prf ldb}.
    We further discuss the lower bounds of Theorem~\ref{thm ldb}(c) and Theorem~\ref{thm ldb}(d) in Section~\ref{proof ldbc} and Section~\ref{prf ldbd} respectively.
    Finally, we conclude in Section~\ref{SecConclusion}.



\section{Proofs of Theorems~\ref{thm lda} and~\ref{thm ldb}}\label{sec proof}
In this section, we present all the proofs. The proofs use similar techniques. The upper bounds are obtained through providing explicit examples. The lower bounds use the discharging method. 
The proof of Theorem~\ref{thm lda} is short, and partially gets implied by Theorem~\ref{thm ldb}. Notice that Theorem~\ref{thm lda}(a) and Theorem~\ref{thm ldb}(a) were already proved 
by Honkala~\cite{honkala2006optimal}, see Theorem~\ref{thmTri1LD}. Thus, our focus is on proving the rest. We present each of the proofs in a separate subsection for the convenience of the readers.

\subsection{Proof of Theorems~\ref{thm lda}(b) and~\ref{thm ldb}(b) (upper bound)}\label{secUpp12bb}
        A $(1, \leq 2)$-LDA code of $\mathbf{G_6}$ 
        with density $\frac{1}{3}$ is given in Figure~\ref{fig:2LDATriagCode-appendix}. Since any $(1, \leq 2)$-LDA code is also a $(1, \leq 2)$-LDB code, this construction provides an upper bound also for $D^B_{1,\ell}$. One may use Remark~\ref{rem density} to verify that the density of the code is indeed $\frac{1}{3}$ by observing the tiling provided in Figure~\ref{fig:2LDATriagCode-appendix}. 
        Specifically, each tile has $6$ vertices, among which $2$ are codewords.

        Next, we show that this is indeed a $(1,\leq 2)$-LDA code. Let us call this code $C$. 
        To verify, let us start with 
        $X,Y\subseteq V(\mathbf{G}_6)$ that satisfy $X \neq Y, X\cap C = Y\cap C$, and $|X|,|Y|\leq 2$. We have to show that $I(X) \neq I(Y)$. 
        Suppose the contrary, that is, $I(X) = I(Y)$. Let us try to analyze such a situation.

        Since $X \neq Y$ and $X \cap C = Y \cap C$, without loss of generality, we may assume that 
        the non-codeword $v_{0,0}\in X$ and $v_{0,0}\notin Y$.
        Thus, $\{ v_{\overline{1},0}, v_{0,1}, v_{1,\overline{1}}\} \subseteq I(X)$.
        Observe that, for any codeword $c$, we have $I(c) = \{c\}$; 
     any two adjacent non-codewords have exactly $2$ codewords in their common neighborhood, and 
    any two non-adjacent non-codewords have 
         at most $1$ codeword in their common neighborhood. 
    Thus,  $I(X) \neq I(Y)$ unless one of the three non-codeword neighbors of $v_{0,0}$ belongs to $Y$.

    Without loss of generality, 
    assume that $v_{1,0} \in Y$. 
    Notice that the codeword 
    $v_{\overline{1},0} \in I(v_{0,0}) \setminus I(v_{1,0})$ and  the codeword $v_{2,0} \in I(v_{1,0}) \setminus I(v_{0,0})$. 
    That means, if $I(X) = I(Y)$, then 
    we must have  
    $X = \{v_{0,0}, y\}$ and 
    $Y = \{v_{1,0}, z\}$  for some 
    $y \in N[v_{2,0}] $ and for some $z \in N[v_{\overline{1},0}] \setminus \{v_{0,0}\}$.
    Observe that $N[v_{\overline{1},0}]\cap N[v_{2,0}] = \emptyset$. 
    Therefore, $y \neq z$. 
    
    \sloppy
    Since both $v_{0,0}$ and $v_{1,0}$ are non-codewords,     the condition $X \cap C = Y \cap C$ implies that $y$ and $z$ are both non-codewords. 
    That means, $y \in \{v_{3,0}, v_{3,\overline{1}}, v_{2,\overline{1}}, v_{1,1}, v_{2,1}, v_{1,0}\}$ and 
    $z \in \{ v_{\overline{2},0}, v_{\overline{2},1}, v_{\overline{1},1}, v_{0,\overline{1}}, v_{\overline{1},\overline{1}} \}$. 
    Hence, at least one element of $A = \{ v_{0,\overline{2}}, v_{\overline{3},1},  v_{\overline{2},2}\}$ must be in $I(Y)$ while $A \cap I(X) = \emptyset$. This is a contradiction to the assumption 
    $I(X) = I(Y)$. Thus,
    $C$ is a $(1,\leq 2)$-LDA code.    
    \qed

        \begin{figure}[H]
        \centering
        \scalebox{.7}{\input{2LDAcodeTriangle.tex}}
        \caption{The gray vertices form a $(1,\leq2)$-LDA code with density $\frac{1}{3}$. }
        \label{fig:2LDATriagCode-appendix}
    \end{figure}

\subsection{Proof of Theorem~\ref{thm lda}(c)}\label{sec1c}
  It is enough to show that every vertex of 
    the infinite triangular grid $\mathbf{G_6}$
    must be part of any $(r, \leq 3)$-LDA code
    as any $(r,\leq \ell+1)$-LDA code is also an $(r,\leq \ell)$-LDA code. Moreover, note that if $C$ is an $(r, \leq 3)$-LDA code, then any superset $C'$ of $C$ is also an $(r, \leq 3)$-LDA code. 
    Thus, without loss of generality, if we can show that $C=V(\mathbf{G_6}) \setminus \{v_{0,0}\}$ is not an $(r, \leq 3)$-LDA code, then we can conclude the proof. 
    Therefore,  consider the code $C$ (see Figure~\ref{Fig:3LDATriangle}).     Notice that 
    $I_r(\{v_{\overline{1},1},v_{1,\overline{1}}\}) = I_r(\{v_{\overline{1},1},v_{1,\overline{1}},v_{0,0}\})$ in this case. Since $C\cap \{v_{\overline{1},1},v_{1,\overline{1}}\}= C\cap \{v_{\overline{1},1},v_{1,\overline{1}},v_{0,0}\}$, code $V(\mathbf{G_6}) \setminus \{v_{0,0}\}$ is not $(r, \leq 3)$-LDA. This concludes the proof.  \qed


    \begin{figure}[H]
        \centering
        \scalebox{0.7}{\begin{tikzpicture}
	\begin{pgfonlayer}{nodelayer}
		\node [style=Empty,fill=Gray!60] (1) at (-7, 0) {};
		\node [style=Empty,fill=Gray!60] (2) at (-6, 0) {};
		\node [style=Empty,fill=Gray!60] (3) at (-5, 0) {};
		\node [style=Empty,fill=Gray!60] (13) at (-7.5, 1) {};
		\node [style=Empty,fill=Gray!60] (14) at (-6.5, 1) {};
		\node [style=Vertex,fill=Gray!60] (15) at (-5.5, 1) {\footnotesize$v_{1,\overline{1}}$};
		\node [style=Empty,fill=Gray!60] (16) at (-4.5, 1) {};
		\node [style=Empty,fill=Gray!60] (25) at (-8, 2) {};
		\node [style=Empty,fill=Gray!60] (26) at (-7, 2) {};
		\node [style=Vertex] (27) at (-6, 2) {\footnotesize$v_{0,0}$};
		\node [style=Empty,fill=Gray!60] (28) at (-5, 2) {};
		\node [style=Empty,fill=Gray!60] (29) at (-4, 2) {};
		\node [style=Empty,fill=Gray!60] (38) at (-7.5, 3) {};
		\node [style=Vertex,fill=Gray!60] (39) at (-6.5, 3) {\footnotesize $v_{\overline{1},1}$};
		\node [style=Empty,fill=Gray!60] (40) at (-5.5, 3) {};
		\node [style=Empty,fill=Gray!60] (41) at (-4.5, 3) {};
		\node [style=Empty,fill=Gray!60] (42) at (-7, 4) {};
		\node [style=Empty,fill=Gray!60] (43) at (-6, 4) {};
		\node [style=Empty,fill=Gray!60] (44) at (-5, 4) {};
	\end{pgfonlayer}
	\begin{pgfonlayer}{edgelayer}

		\draw (13) to (26);
		\draw (26) to (39);
		\draw (40) to (27);
		\draw (27) to (14);
		\draw (14) to (1);
		\draw (1) to (2);
		\draw (2) to (3);
		\draw (16) to (15);
		\draw (15) to (14);
		\draw (13) to (14);
		\draw (13) to (25);
		\draw (25) to (26);
		\draw (26) to (27);
		\draw (27) to (28);
		\draw (29) to (28);
		\draw (41) to (40);
		\draw (40) to (39);
		\draw (39) to (38);
		\draw (29) to (16);
		\draw (16) to (3);
		\draw (3) to (15);
		\draw (15) to (27);
		\draw (27) to (39);
		\draw (40) to (28);
		\draw (28) to (16);
		\draw (2) to (15);
		\draw (15) to (28);
		\draw (28) to (41);
		\draw (41) to (29);
		\draw (2) to (14);
		\draw (14) to (26);
		\draw (26) to (38);
		\draw (38) to (25);
		\draw (13) to (1);
		\draw (38) to (42);
		\draw (42) to (39);
		\draw (39) to (43);
		\draw (43) to (40);
		\draw (40) to (44);
		\draw (44) to (41);
		\draw (44) to (43);
		\draw (43) to (42);
	\end{pgfonlayer}
\end{tikzpicture}}
        \caption{A reference figure for the proof of Theorem~\ref{thm lda}(c) where gray vertices are codewords.}
        \label{Fig:3LDATriangle}
    \end{figure}

\subsection{Proof of Theorem~\ref{thm ldb}(c) (upper bound)}\label{sec2cupp}

A $(1, \leq 3)$-LDB code of $\mathbf{G_6}$ 
        with density $\frac{4}{7}$ is given in Figure~\ref{fig:3LDBFig}. 
        One may use Remark~\ref{rem density} to verify that the density of the code is indeed $\frac{4}{7}$ by observing the tiling provided in Figure~\ref{fig:3LDBFig}. 
        Specifically, each hexagonal tile has $7$ vertices, among which $4$ are codewords.

        Next, we show that this is indeed a $(1,\leq 3)$-LDB code. Let us call this code $C$. 
        To verify, consider 
        $X,Y\subseteq V(\mathbf{G}_6)\setminus C$ that satisfy $X \neq Y$, and $|X|,|Y|\leq 3$. We have to show that $I(X) \neq I(Y)$. 
        Suppose the contrary, that is, $I(X) = I(Y)$. 

        Since $X\neq Y$, without loss of generality, suppose that the non-codeword $v_{0,0}\in X$, but $v_{0,0}\notin Y$.
        Thus, $A = \{ v_{\overline{1},1}, v_{\overline{1},0}, v_{0,\overline{1}}, v_{1,\overline{1}} \} \subseteq I(X)$.
        Since $v_{0,0} \not\in Y$, and $I(X) = I(Y)$, we must have at least one vertex from each of the following vertex subsets in $Y$: 
        $A_1 = \{ v_{\overline{2},0}, v_{\overline{1},\overline{1}} \}$ (as $v_{\overline{1}, 0} \in I(Y)$),  
        $A_2 = \{ v_{\overline{2},2}, v_{0,1} \}$
        (as $v_{\overline{1}, 1} \in I(Y)$), 
        $A_3 = \{ v_{1,0}, v_{1,\overline{2}} \}$
        (as $v_{1, \overline{1}} \in I(Y)$). 
        That forces $|Y| = 3$, 
        and $Y \subseteq A_1 \cup A_2 \cup A_3$. 
        In particular, that implies that $I(X)$ cannot contain any codeword not adjacent to the vertices of $A_1 \cup A_2 \cup A_3$ due to our assumption $I(X) = I(Y)$. 
        Observe that any non-codeword $x \not\in A_1 \cup A_2 \cup A_3 \cup \{v_{0,0}\}$ has a codeword neighbor that is not adjacent to any vertex of  $A_1 \cup A_2 \cup A_3$. That means, $X \subseteq A_1 \cup A_2 \cup A_3 \cup \{v_{0,0}\}$. 
        Further observe that, 
        each vertex $x \in A_1 \cup A_2 \cup A_3$ has a codeword neighbor that is not adjacent to any 
        vertex of $(A_1 \cup A_2 \cup A_3 \cup \{v_{0,0}\}) \setminus \{x\}$. 
        Hence, if a vertex $x \in A_1 \cup A_2 \cup A_3$ belongs to $Y$, then it must also belong to $X$. This implies $Y \subseteq X$ forcing $|X| \geq |Y \cup \{v_{0,0}\}| \geq 3+1 = 4$, a contradiction.\qed

    \begin{figure}[H]
        \centering
        \scalebox{.8}{\input{3LDBTriangle}}
        \caption{The gray vertices form a $(1,\leq3)$-LDB code with density $\frac{4}{7}$.}
        \label{fig:3LDBFig}
    \end{figure}

\subsection{Proof of Theorem~\ref{thm ldb}(d) (upper bound)}\label{sec2dupp}
 A $(1, \leq 4)$-LDB code of $\mathbf{G_6}$ 
        with density $\frac{2}{3}$ is given in Figure~\ref{fig:4LDBcodeTriangle}. 
        One may use Remark~\ref{rem density} to verify that the density of the code is indeed $\frac{2}{3}$ by observing the tiling provided in Figure~\ref{fig:4LDBcodeTriangle}. 
        Specifically, each tile has $6$ vertices, among which $4$ are codewords.
        
        The given construction is a $(1, \leq 4)$-LDB code since each non-codeword is adjacent to a codeword which is adjacent to no other non-codewords. For example, if we have $v_{0,0}\in I(X)$, then we immediately know that $v_{0,1}\in X$.
        This establishes the upper bound $D^B_{1,4} \leq \frac{2}{3}$. 
        \qed

\begin{figure}[H]
        \centering
        \scalebox{.8}{\input{4LDBcodeTriangle}}
        \caption{The gray vertices form a $(1,\leq4)$-LDB code with density $\frac{2}{3}$.}
        \label{fig:4LDBcodeTriangle}
    \end{figure}

\subsection{A general discharging method for proving the lower bounds} \label{subsection general technique} 
In this section, we establish the lower bound proof technique which we use to show the lower bounds of Theorems~\ref{thm ldb}(b), (c), and (d). To avoid duplicating certain specifics for those proofs, we are presenting a general overview of them here. Our proofs are based on the discharging method. An interested reader may see~\cite{cranston2017introduction} to read more on the discharging method.

Let $C$ be a $(1, \leq \ell)$-LDB code of $\mathbf{G_6}$.  We wish to establish a lower bound $D^B_{1,\ell} \geq \alpha$, for some $\alpha \in (0, 1)$. The proof strategy is as follows. Initially, we assign a charge function
to the vertices of $\mathbf{G_6}$ as follows:
$$
ch(x) =\begin{cases}
			1, & \text{if } x\in C\\
            0, & \text{if } x \in V(\mathbf{G_6}) \setminus C.
		 \end{cases}
$$

In the next step, we define certain \textit{discharging rules} and relocate the charges. In this way, the total charge of the graph is conserved. 
We use the following  discharging rules in each of our proofs. Thus, let us present the discharging rules here as a common step of the proofs.

\begin{itemize}
    \item[(R1)] Each codeword $c \in C$ with $|I(c)|<7$ donates a charge of $\frac{1}{7-|I(c)|}$ to its non-codeword neighbors. 
    The updated charge of a vertex $x$ after applying (R1) is denoted by $ch_1(x)$. 

    \item[(R2)] Suppose a non-codeword $v$ has charge $ch_1(v) = \frac{\alpha}{1 - \alpha}+\beta$, for some $\beta > 0$, after (R1). If $v$ has $t>0$
    non-codeword neighbors 
    $u_1, u_2, \dots, u_t$ with $ch_1(u_i) < \frac{\alpha}{1 - \alpha}$, for all  $i \in \{1,2, \ldots, t\}$, then $v$ donates a charge of $\frac{\beta}{t}$ to each of them.  The updated charge of the vertex $x$, after applying (R2), is denoted by $ch_2(x)$. 
\end{itemize}
Note that we apply rule (R1) simultaneously to all the vertices. After that, we apply rule (R2) to every vertex simultaneously.

Suppose that after applying (R1) and (R2) we have 
$$
ch_2(x) \geq \begin{cases}
			0, & \text{if } x\in C\\
            \frac{\alpha}{1 - \alpha}, & \text{if } x \in V(\mathbf{G_6}) \setminus C.
		 \end{cases}
$$
By Lemma~\ref{lem lower bound}, this will prove that the density of $C$ is at least $\alpha$, and hence prove the lower bound $D^B_{1,\ell} \geq \alpha$.

Let $c$ be a codeword, that is, $c \in C$. 
If $c$ has $k \geq 1$ non-codewords in its neighborhood, then $c$ donates a total charge of $$k \times \frac{1}{7-|I(c)|} = k \times \frac{1}{7 - (7-k)} = 1$$ 
        due to (R1). That means, $ch_1(c) \geq 1-1 =0$. 
        Moreover, application of (R2) does not change the charge of any codewords. 
        This implies that after applying (R1) and (R2), the updated charge of each codeword $c$ is $ch_1(c) = ch_2(c) \geq 0$.

Furthermore, let $u$ be a non-codeword. Notice that, if $ch_1(u) \geq \frac{\alpha}{1 - \alpha}$ then, 
even after applying (R2), we will still have 
$ch_2(u) \geq \frac{\alpha}{1 - \alpha}$. 

\begin{lemma}\label{lem lower bound}
If $ch_2(u) \geq \frac{\alpha}{1 - \alpha}$, for all $u \in V(\mathbf{G_6}) \setminus C$, 
then $D^B_{1,\ell} \geq \alpha$. 
\end{lemma}
\begin{proof}
    Observe that we move charge to at most distance 2 from a codeword. Hence, when $ch_2(u) \geq \frac{\alpha}{1 - \alpha}$, for all $u \in V(\mathbf{G_6}) \setminus C$, we obtain
    $|C\cap B_n(v_{0,0})| \geq \frac{\alpha}{1-\alpha}(|B_{n-2}(v_{0,0})|-|C \cap B_{n-2}(v_{0,0})|) \geq \frac{\alpha}{1-\alpha}(|B_{n-2}(v_{0,0})|-|C\cap B_n(v_{0,0})|)$. This implies $|C\cap B_n(v_{0,0})|\geq |B_{n-2}(v_{0,0})|\alpha$. Therefore, we have $$\limsup_{n \rightarrow \infty} \frac{|C \cap B_n(v_{0,0})|}{|B_n(v_{0,0})|}\geq \limsup_{n \rightarrow \infty}\frac{|B_{n-2}(v_{0,0})|\alpha}{|B_n(v_{0,0})|}=\alpha \limsup_{n \rightarrow \infty}\frac{|B_{n-2}(v_{0,0})|}{|B_n(v_{0,0})|}=\alpha. \hspace*{\fill}\qedhere$$
\end{proof}

\begin{remark}
    Notice that, the above process and Lemma~\ref{lem lower bound} can be generalized further by changing the discharging rules. Moreover, it is possible to replace $\mathbf{G_6}$ by a square, hexagonal, or a King's grid keeping the method the same (as long as value 7 in rule (R1) is chosen well). 
    Furthermore, the technique can be readily extended to prove lower bounds of the minimum density of
    $(r, \leq \ell)$-LDA and $(r, \leq \ell)$-LDB codes for general values of $r$ and $\ell$. Similar discharging idea has been used in~\cite{das2026density}.  
    
    Another similar concept is share, introduced in~\cite{slater2002fault}, which has been used widely in lower bound proofs of densities. However, for us share did not seem to work as well as our own approach which a bit surprisingly is working in every case using the same framework of two rules.
\end{remark}

\subsection{Proof of Theorems~\ref{thm lda}(b) and~\ref{thm ldb}(b) (lower bound)}\label{prf ldb}
In this section, we show that $D^A_{1,2}\geq D^B_{1,2}\geq\frac{1}{3}$.
Recall that, any $(1,\leq 2)$-LDA code is, in particular, a  $(1,\leq 2)$-LDB code. Thus, it is enough to prove that any $(1,\leq 2)$-LDB code has density at least $\frac{1}{3}$. 
It is enough to show that 
     $ch_2(v_{0,0}) \geq \frac{1}{2}=(1/3)/(2/3)$ assuming $v_{0,0}$ is a non-codeword due to Lemma~\ref{lem lower bound} along with the discussion in Section~\ref{subsection general technique}.

       Hence, let us calculate $ch_2(u)$ for  each 
       non-codeword $u$  based on how many adjacent codewords they have, and their positions. 
       For this part, without loss of generality, we will assume $v_{0,0}$ to be a non-codeword, and consider all the cases under this assumption. That is, under different scenarios, we want to show that $ch_2(v_{0,0}) \geq \frac{1}{2}$, which will conclude the proof. 

       \noindent\textbf{Case 1: $\mathbf{|I(v_{0,0})| \geq 3}$}.   
       Note that, every codeword donates a charge of 
       at least $\frac{1}{6}$ to $v_{0,0}$. Thus,  $ch_1(v_{0,0}) \geq 3 \times \frac{1}{6} = \frac{1}{2}$ as $|I(v_{0,0})| \geq 3$.
       Therefore, $ch_2(v_{0,0}) \geq \frac{1}{2}$. 

\medskip 
       
       \noindent\textbf{Case 2: $\mathbf{|I(v_{0,0})| = 2}$.}  Suppose that $|N(v_{0,0}) \cap C| = 2$.
       Observe that, the two codeword neighbors of $v_{0,0}$ can be at distance $1$, $2$ or $3$ in $\mathbf{G_6}[N(v_{0,0})]$, that is, the $6$-cycle induced by the neighbors  of $v_{0,0}$.
       
\medskip
       
 \noindent \textbf{Subcase 2.1: the two codeword neighbors of $\mathbf{v_{0,0}}$ are at distance $\mathbf{1}$ in $\mathbf{\mathbf{G_6}[N(v_{0,0})]}$.}      Without loss of generality, suppose that the two codeword neighbors of $v_{0,0}$ are $v_{0,1}$ and $v_{1,0}$ (depicted in Figure~\ref{fig 3ldb}(a) for reference). That is, capturing the scenario when the two codeword neighbors are at distance $1$ in 
       $\mathbf{G_6}[N(v_{0,0})]$.
       In this scenario, 
       $I(\{ v_{0,0},v_{1,1} \})=I(v_{1,1})$, and thus, $v_{1,1}$ must be a codeword. That implies $|I(v_{0,1})|, |I(v_{1,0})| \geq 3$. Therefore, both $v_{0,1}$ and $v_{1,0}$ donate 
       $\frac{1}{4}$ each to $v_{0,0}$. Thus, $ch_1(v_{0,0}) \geq \frac{1}{2}$, which implies $ch_2(v_{0,0}) \geq \frac{1}{2}$.

       \begin{figure}[h]
            \centering
            \begin{subfigure}[b]{0.4\textwidth}
                \centering
                \scalebox{.8}{\begin{tikzpicture}
	\begin{pgfonlayer}{nodelayer}
		

\node[style=Vertex] (v-2-2) at (-1.5,-2) {\footnotesize$v_{\overline{1},\overline{2}}$};
\node[style=Vertex] (v-1-2) at (-0.5,-2) {\footnotesize$v_{0,\overline{2}}$};
\node[style=Vertex] (v0-2)  at ( 0.5,-2) {\footnotesize$v_{1,\overline{2}}$};
\node[style=Vertex] (v1-2)  at ( 1.5,-2) {\footnotesize$v_{2,\overline{2}}$};
\node[style=Vertex] (v2-2)  at ( 2.5,-2) {\footnotesize$v_{3,\overline{2}}$};

\node[style=Vertex] (v-2-1) at (-2,-1) {\footnotesize$v_{\overline{2},\overline{1}}$};
\node[style=Vertex] (v-1-1) at (-1,-1) {\footnotesize$v_{\overline{1},\overline{1}}$};
\node[style=Vertex] (v0-1)  at ( 0,-1) {\footnotesize$v_{0,\overline{1}}$};
\node[style=Vertex] (v1-1)  at ( 1,-1) {\footnotesize$v_{1,\overline{1}}$};
\node[style=Vertex] (v2-1)  at ( 2,-1) {\footnotesize$v_{2,\overline{1}}$};

\node[style=Vertex] (v-20) at (-1.5,0) {\footnotesize$v_{\overline{2},0}$};
\node[style=Vertex] (v-10) at (-0.5,0) {\footnotesize$v_{\overline{1},0}$};
\node[style=Vertex] (v00)   at ( 0.5,0) {\footnotesize$v_{0,0}$};
\node[style=Vertex, fill=Gray!60] (v10)   at ( 1.5,0) {\footnotesize$v_{1,0}$};
\node[style=Vertex] (v20)   at ( 2.5,0) {\footnotesize$v_{2,0}$};

\node[style=Vertex] (v-21) at (-2,1) {\footnotesize$v_{\overline{3},1}$};
\node[style=Vertex] (v-11) at (-1,1) {\footnotesize$v_{\overline{2},1}$};
\node[style=Vertex] (v01)   at ( 0,1) {\footnotesize$v_{\overline{1},1}$};
\node[style=Vertex,fill=Gray!60] (v11)   at ( 1,1) {\footnotesize$v_{0,1}$};
\node[style=Vertex,fill=Gray!60] (v21)   at ( 2,1) {\footnotesize$v_{1,1}$};

\node[style=Vertex] (v-22) at (-1.5,2) {\footnotesize$v_{\overline{3},2}$};
\node[style=Vertex] (v-12) at (-0.5,2) {\footnotesize$v_{\overline{2},2}$};
\node[style=Vertex] (v02)   at ( 0.5,2) {\footnotesize$v_{\overline{1},2}$};
\node[style=Vertex] (v12)   at ( 1.5,2) {\footnotesize$v_{0,2}$};
\node[style=Vertex] (v22)   at ( 2.5,2) {\footnotesize$v_{1,2}$};

        
	\end{pgfonlayer}
    
	\begin{pgfonlayer}{edgelayer}

        \foreach \j in {-2,...,2}{
            \foreach \i in {-2,...,1}{ 
                \draw (v\i\j) -- (v\the\numexpr\i+1\relax\j);
            }
        }

        \foreach \j in {-1,1}{
            \foreach \i in {-2,...,2}{
                \draw (v\i\j) -- (v\i\the\numexpr\j+1\relax);
            }
            \foreach \i in {-1,...,2}{ 
                \draw (v\i\j) -- (v\the\numexpr\i-1\relax\the\numexpr\j+1\relax);
            }
        }
        
        \foreach \j in {-2,0}{
            \foreach \i in {-2,...,2}{
                \draw (v\i\j) -- (v\i\the\numexpr\j+1\relax);
            }
            \foreach \i in {-2,...,1}{ 
                \draw (v\i\j) -- (v\the\numexpr\i+1\relax\the\numexpr\j+1\relax);
            }
        }


	\end{pgfonlayer}
\end{tikzpicture}}
                \caption{}
            \end{subfigure}
            \hfill
            \begin{subfigure}[b]{0.4\textwidth}
                \centering
                \scalebox{.8}{\begin{tikzpicture}
	\begin{pgfonlayer}{nodelayer}
		

\node[style=Vertex] (v-2-2) at (-1.5,-2) {\footnotesize$v_{\overline{1},\overline{2}}$};
\node[style=Vertex] (v-1-2) at (-0.5,-2) {\footnotesize$v_{0,\overline{2}}$};
\node[style=Vertex] (v0-2)  at ( 0.5,-2) {\footnotesize$v_{1,\overline{2}}$};
\node[style=Vertex] (v1-2)  at ( 1.5,-2) {\footnotesize$v_{2,\overline{2}}$};
\node[style=Vertex] (v2-2)  at ( 2.5,-2) {\footnotesize$v_{3,\overline{2}}$};

\node[style=Vertex] (v-2-1) at (-2,-1) {\footnotesize$v_{\overline{2},\overline{1}}$};
\node[style=Vertex] (v-1-1) at (-1,-1) {\footnotesize$v_{\overline{1},\overline{1}}$};
\node[style=Vertex] (v0-1)  at ( 0,-1) {\footnotesize$v_{0,\overline{1}}$};
\node[style=Vertex] (v1-1)  at ( 1,-1) {\footnotesize$v_{1,\overline{1}}$};
\node[style=Vertex] (v2-1)  at ( 2,-1) {\footnotesize$v_{2,\overline{1}}$};

\node[style=Vertex] (v-20) at (-1.5,0) {\footnotesize$v_{\overline{2},0}$};
\node[style=Vertex] (v-10) at (-0.5,0) {\footnotesize$v_{\overline{1},0}$};
\node[style=Vertex] (v00)   at ( 0.5,0) {\footnotesize$v_{0,0}$};
\node[style=Vertex, fill=Gray!60] (v10)   at ( 1.5,0) {\footnotesize$v_{1,0}$};
\node[style=Vertex] (v20)   at ( 2.5,0) {\footnotesize$v_{2,0}$};

\node[style=Vertex] (v-21) at (-2,1) {\footnotesize$v_{\overline{3},1}$};
\node[style=Vertex] (v-11) at (-1,1) {\footnotesize$v_{\overline{2},1}$};
\node[style=Vertex,fill=Gray!60] (v01)   at ( 0,1) {\footnotesize$v_{\overline{1},1}$};
\node[style=Vertex,fill=Gray!60] (v11)   at ( 1,1) {\footnotesize$v_{0,1}$};
\node[style=Vertex] (v21)   at ( 2,1) {\footnotesize$v_{1,1}$};

\node[style=Vertex] (v-22) at (-1.5,2) {\footnotesize$v_{\overline{3},2}$};
\node[style=Vertex] (v-12) at (-0.5,2) {\footnotesize$v_{\overline{2},2}$};
\node[style=Vertex] (v02)   at ( 0.5,2) {\footnotesize$v_{\overline{1},2}$};
\node[style=Vertex] (v12)   at ( 1.5,2) {\footnotesize$v_{0,2}$};
\node[style=Vertex] (v22)   at ( 2.5,2) {\footnotesize$v_{1,2}$};

        
	\end{pgfonlayer}
    
	\begin{pgfonlayer}{edgelayer}

        \foreach \j in {-2,...,2}{
            \foreach \i in {-2,...,1}{ 
                \draw (v\i\j) -- (v\the\numexpr\i+1\relax\j);
            }
        }

        \foreach \j in {-1,1}{
            \foreach \i in {-2,...,2}{
                \draw (v\i\j) -- (v\i\the\numexpr\j+1\relax);
            }
            \foreach \i in {-1,...,2}{ 
                \draw (v\i\j) -- (v\the\numexpr\i-1\relax\the\numexpr\j+1\relax);
            }
        }
        
        \foreach \j in {-2,0}{
            \foreach \i in {-2,...,2}{
                \draw (v\i\j) -- (v\i\the\numexpr\j+1\relax);
            }
            \foreach \i in {-2,...,1}{ 
                \draw (v\i\j) -- (v\the\numexpr\i+1\relax\the\numexpr\j+1\relax);
            }
        }


	\end{pgfonlayer}
\end{tikzpicture}}
                \caption{}
            \end{subfigure}
            \begin{subfigure}[b]{0.4\textwidth}
            \centering
            \scalebox{.8}{\begin{tikzpicture}
	\begin{pgfonlayer}{nodelayer}
		

\node[style=Vertex] (v-2-2) at (-1.5,-2) {\footnotesize$v_{\overline{1},\overline{2}}$};
\node[style=Vertex] (v-1-2) at (-0.5,-2) {\footnotesize$v_{0,\overline{2}}$};
\node[style=Vertex] (v0-2)  at ( 0.5,-2) {\footnotesize$v_{1,\overline{2}}$};
\node[style=Vertex] (v1-2)  at ( 1.5,-2) {\footnotesize$v_{2,\overline{2}}$};
\node[style=Vertex] (v2-2)  at ( 2.5,-2) {\footnotesize$v_{3,\overline{2}}$};

\node[style=Vertex] (v-2-1) at (-2,-1) {\footnotesize$v_{\overline{2},\overline{1}}$};
\node[style=Vertex] (v-1-1) at (-1,-1) {\footnotesize$v_{\overline{1},\overline{1}}$};
\node[style=Vertex] (v0-1)  at ( 0,-1) {\footnotesize$v_{0,\overline{1}}$};
\node[style=Vertex] (v1-1)  at ( 1,-1) {\footnotesize$v_{1,\overline{1}}$};
\node[style=Vertex] (v2-1)  at ( 2,-1) {\footnotesize$v_{2,\overline{1}}$};

\node[style=Vertex] (v-20) at (-1.5,0) {\footnotesize$v_{\overline{2},0}$};
\node[style=Vertex, fill=Gray!60] (v-10) at (-0.5,0) {\footnotesize$v_{\overline{1},0}$};
\node[style=Vertex] (v00)   at ( 0.5,0) {\footnotesize$v_{0,0}$};
\node[style=Vertex, fill=Gray!60] (v10)   at ( 1.5,0) {\footnotesize$v_{1,0}$};
\node[style=Vertex] (v20)   at ( 2.5,0) {\footnotesize$v_{2,0}$};

\node[style=Vertex] (v-21) at (-2,1) {\footnotesize$v_{\overline{3},1}$};
\node[style=Vertex] (v-11) at (-1,1) {\footnotesize$v_{\overline{2},1}$};
\node[style=Vertex] (v01)   at ( 0,1) {\footnotesize$v_{\overline{1},1}$};
\node[style=Vertex] (v11)   at ( 1,1) {\footnotesize$v_{0,1}$};
\node[style=Vertex] (v21)   at ( 2,1) {\footnotesize$v_{1,1}$};

\node[style=Vertex] (v-22) at (-1.5,2) {\footnotesize$v_{\overline{3},2}$};
\node[style=Vertex] (v-12) at (-0.5,2) {\footnotesize$v_{\overline{2},2}$};
\node[style=Vertex] (v02)   at ( 0.5,2) {\footnotesize$v_{\overline{1},2}$};
\node[style=Vertex] (v12)   at ( 1.5,2) {\footnotesize$v_{0,2}$};
\node[style=Vertex] (v22)   at ( 2.5,2) {\footnotesize$v_{1,2}$};

        
	\end{pgfonlayer}
    
	\begin{pgfonlayer}{edgelayer}

        \foreach \j in {-2,...,2}{
            \foreach \i in {-2,...,1}{ 
                \draw (v\i\j) -- (v\the\numexpr\i+1\relax\j);
            }
        }

        \foreach \j in {-1,1}{
            \foreach \i in {-2,...,2}{
                \draw (v\i\j) -- (v\i\the\numexpr\j+1\relax);
            }
            \foreach \i in {-1,...,2}{ 
                \draw (v\i\j) -- (v\the\numexpr\i-1\relax\the\numexpr\j+1\relax);
            }
        }
        
        \foreach \j in {-2,0}{
            \foreach \i in {-2,...,2}{
                \draw (v\i\j) -- (v\i\the\numexpr\j+1\relax);
            }
            \foreach \i in {-2,...,1}{ 
                \draw (v\i\j) -- (v\the\numexpr\i+1\relax\the\numexpr\j+1\relax);
            }
        }


	\end{pgfonlayer}
\end{tikzpicture}}
            \caption{}
        \end{subfigure}
        \hfill
        \begin{subfigure}[b]{0.4\textwidth}
            \centering
            \scalebox{.8}{\begin{tikzpicture}
	\begin{pgfonlayer}{nodelayer}
		

\node[style=Vertex] (v-2-2) at (-1.5,-2) {\footnotesize$v_{\overline{1},\overline{2}}$};
\node[style=Vertex] (v-1-2) at (-0.5,-2) {\footnotesize$v_{0,\overline{2}}$};
\node[style=Vertex] (v0-2)  at ( 0.5,-2) {\footnotesize$v_{1,\overline{2}}$};
\node[style=Vertex] (v1-2)  at ( 1.5,-2) {\footnotesize$v_{2,\overline{2}}$};
\node[style=Vertex] (v2-2)  at ( 2.5,-2) {\footnotesize$v_{3,\overline{2}}$};

\node[style=Vertex] (v-2-1) at (-2,-1) {\footnotesize$v_{\overline{2},\overline{1}}$};
\node[style=Vertex] (v-1-1) at (-1,-1) {\footnotesize$v_{\overline{1},\overline{1}}$};
\node[style=Vertex] (v0-1)  at ( 0,-1) {\footnotesize$v_{0,\overline{1}}$};
\node[style=Vertex] (v1-1)  at ( 1,-1) {\footnotesize$v_{1,\overline{1}}$};
\node[style=Vertex] (v2-1)  at ( 2,-1) {\footnotesize$v_{2,\overline{1}}$};

\node[style=Vertex] (v-20) at (-1.5,0) {\footnotesize$v_{\overline{2},0}$};
\node[style=Vertex, fill=Gray!60] (v-10) at (-0.5,0) {\footnotesize$v_{\overline{1},0}$};
\node[style=Vertex] (v00)   at ( 0.5,0) {\footnotesize$v_{0,0}$};
\node[style=Vertex, fill=Gray!60] (v10)   at ( 1.5,0) {\footnotesize$v_{1,0}$};
\node[style=Vertex] (v20)   at ( 2.5,0) {\footnotesize$v_{2,0}$};

\node[style=Vertex] (v-21) at (-2,1) {\footnotesize$v_{\overline{3},1}$};
\node[style=Vertex] (v-11) at (-1,1) {\footnotesize$v_{\overline{2},1}$};
\node[style=Vertex] (v01)   at ( 0,1) {\footnotesize$v_{\overline{1},1}$};
\node[style=Vertex] (v11)   at ( 1,1) {\footnotesize$v_{0,1}$};
\node[style=Vertex] (v21)   at ( 2,1) {\footnotesize$v_{1,1}$};

\node[style=Vertex] (v-22) at (-1.5,2) {\footnotesize$v_{\overline{3},2}$};
\node[style=Vertex] (v-12) at (-0.5,2) {\footnotesize$v_{\overline{2},2}$};
\node[style=Vertex,fill=Gray!60] (v02)   at ( 0.5,2) {\footnotesize$v_{\overline{1},2}$};
\node[style=Vertex,fill=Gray!60] (v12)   at ( 1.5,2) {\footnotesize$v_{0,2}$};
\node[style=Vertex] (v22)   at ( 2.5,2) {\footnotesize$v_{1,2}$};

        
	\end{pgfonlayer}
    
	\begin{pgfonlayer}{edgelayer}

        \foreach \j in {-2,...,2}{
            \foreach \i in {-2,...,1}{ 
                \draw (v\i\j) -- (v\the\numexpr\i+1\relax\j);
            }
        }

        \foreach \j in {-1,1}{
            \foreach \i in {-2,...,2}{
                \draw (v\i\j) -- (v\i\the\numexpr\j+1\relax);
            }
            \foreach \i in {-1,...,2}{ 
                \draw (v\i\j) -- (v\the\numexpr\i-1\relax\the\numexpr\j+1\relax);
            }
        }
        
        \foreach \j in {-2,0}{
            \foreach \i in {-2,...,2}{
                \draw (v\i\j) -- (v\i\the\numexpr\j+1\relax);
            }
            \foreach \i in {-2,...,1}{ 
                \draw (v\i\j) -- (v\the\numexpr\i+1\relax\the\numexpr\j+1\relax);
            }
        }


	\end{pgfonlayer}
\end{tikzpicture}}
            \caption{}
        \end{subfigure}

            \caption{Reference figures for Case 2 in the proof of Theorem~\ref{thm lda}(b) and Theorem~\ref{thm ldb}(b), where gray vertices are codewords.}
            \label{fig 3ldb}
        \end{figure}

    \medskip
       
 \noindent \textbf{Subcase 2.2:  the two codeword neighbors of $\mathbf{v_{0,0}}$ are at distance $\mathbf{2}$ in $\mathbf{\mathbf{G_6}[N(v_{0,0})]}$.}
       Without loss of generality, suppose that the two codeword neighbors of $v_{0,0}$ are $v_{\overline{1},1}$ and $v_{1,0}$ (depicted in Figure~\ref{fig 3ldb}(b) for reference). 
       That is, capturing the scenario when the two codeword neighbors are at distance $2$ in $\mathbf{G_6}[N(v_{0,0})]$, without loss of generality. 
       In this scenario, 
       $I(\{ v_{0,0},v_{0,1} \})=I(v_{0,1})$, and thus, $v_{0,1}$ must be a codeword, contradicting the fact that $v_{0,0}$ has exactly two codeword neighbors. Thus, if $v_{0,0}$ has exactly two codeword neighbors, it is not possible for them to be at distance  $2$ 
       in $\mathbf{G_6}[N(v_{0,0})]$. 
       In general, we can conclude that 
       a non-codeword $u$ cannot have exactly two codeword neighbors
       at distance $2$ in 
       $\mathbf{G_6}[N(u)]$.

\medskip
       
 \noindent \textbf{Subcase 2.3:  the two codeword neighbors of $\mathbf{v_{0,0}}$ are at distance $\mathbf{3}$ in $\mathbf{\mathbf{G_6}[N(v_{0,0})]}$.}
 Finally, without loss of generality, suppose that the two codeword neighbors of $v_{0,0}$ are $v_{1,0}$ and $v_{\overline{1}, 0}$(depicted in Figure~\ref{fig 3ldb}(c) for reference). 
       That is, capturing the scenario when the two codeword neighbors are at distance $3$ in $\mathbf{G_6}[N(v_{0,0})]$, without loss of generality. We will analyze this scenario further, starting with a claim.

\medskip

\noindent\textit{Claim 1.}\quad
    If $v_{1,1}$ (resp., $v_{\overline{2},1}, v_{\overline{1},\overline{1}}, v_{2,\overline{1}}$) is a non-codeword, then $v_{0,1}$ (resp., $v_{\overline{1},1},$ $v_{0,\overline{1}},$ $v_{1,\overline{1}}$) donates a charge of 
    $\frac{1}{15}$ to $v_{0,0}$ according to (R2).

\noindent \textbf{Proof of Claim:}
    Observe that, it is enough to 
    show that if $v_{1,1}$ 
    is a non-codeword, then $v_{0,1}$ 
    donates a charge of 
    $\frac{1}{15}$ to $v_{0,0}$ according to (R2) due to symmetry. 
    Notice that, if $v_{1,1}$ is a non-codeword, 
    then $v_{0, 2}$ is a codeword, 
       as otherwise $I(\{v_{0,0}, v_{\overline{1},1}\}) = I(\{v_{0,1}, v_{\overline{1},1}\})$. Now, notice that, $v_{0,1}$ has two codeword neighbors at distance exactly $2$ in 
       $\mathbf{G_6}[N(v_{0,1})]$, unless $v_{\overline{1}, 2}$ is also a codeword. 
       Recall that, in \textbf{Subcase 2.2} we had established that 
       if a non-codeword $u$ has exactly two codeword neighbors, then they cannot be at distance exactly $2$  from each other in the $6$-cycle 
       $\mathbf{G_6}[N(u)]$. Hence,  $v_{\overline{1}, 2}$ is also a codeword. See Figure~\ref{fig 3ldb}(d) for reference. 
Therefore, $ch_1(v_{0,1}) \geq \frac{1}{5} + \frac{1}{5} + \frac{1}{6} = \frac{17}{30} = \frac{1}{2} + \frac{1}{15}$.  

Note that, $v_{\overline{1},1}$ has (at least) two codeword neighbors: $v_{\overline{1},2}$ and $v_{\overline{1},0}$. Moreover, if $v_{\overline{2}, 1}$ is a non-codeword, then 
$v_{\overline{2}, 2}$ must be a codeword as  $I(\{v_{0,0},v_{0,1}\}) =I(\{v_{\overline{1},1},v_{0,1}\})$. That means, either $v_{\overline{2}, 1}$ or 
$v_{\overline{2}, 2}$ must be a codeword. This will force $v_{\overline{1},1}$ to have at least three codeword neighbors, and thus $ch_1(v_{\overline{1},1}) \geq \frac{1}{2}$ 
due to \textbf{Case 1}. Hence, $v_{0,1}$ does not donate any charge to $v_{\overline{1},1}$ 
according to (R2). 

Next, observe that, $v_{1,1}$ has two 
codeword neighbors $v_{1,0}$ and $v_{0,2}$. 
Due to \textbf{Case 2}, it is not possible to have $|I(v_{1,0})| = 2$ since 
$v_{1,0}$ and $v_{0,2}$ are at distance $2$ in 
$\mathbf{G_6}[N(v_{1,1})]$. 
This will force $v_{1,1}$ to have at least three codeword neighbors, and thus $ch_1(v_{1,1}) \geq \frac{1}{2}$ 
due to \textbf{Case 1}.
Hence, $v_{0,1}$ does not donate any charge to $v_{1,1}$ 
according to (R2). 

That means, $v_{0,1}$ donates a charge of at least
    $\frac{1}{15}$ to $v_{0,0}$ according to (R2). 
    \hfill $\circ$   
       
\medskip

If both $v_{1,1}$ and $v_{2,\overline{1}}$ are codewords, 
then $|I(v_{1,0})| \geq 3$
and $v_{1,0}$ donates at least $\frac{1}{4}$ to $v_{0,0}$ according to (R1).

If exactly one of $v_{1,1}$ and $v_{2,\overline{1}}$ are codewords, 
then $|I(v_{1,0})| \geq 2$
and $v_{1,0}$ donates at least $\frac{1}{5}$ to $v_{0,0}$ according to (R1). 
Moreover, one of $v_{0,1}$ and $v_{1, \overline{1}}$ donates $\frac{1}{15}$ to $v_{0,0}$ according to (R2) due to Claim 1. 
That means,  $v_{1,0}, v_{0,1}, v_{1,\overline{1}}$ collectively donate a charge of at least $\frac{1}{5} + \frac{1}{15} = \frac{4}{15} > \frac{1}{4}$.

If both $v_{1,1}$ and $v_{2,\overline{1}}$ are non-codewords, 
then $|I(v_{1,0})| \geq 1$ 
and  $v_{1,0}$ donates $\frac{1}{6}$ to $v_{0,0}$ according to (R1).
Moreover, each of $v_{0,1}$ and $v_{1, \overline{1}}$ donates at least $\frac{1}{15}$ to $v_{0,0}$ according to (R2) due to Claim 1. 
That implies,  $v_{1,0}, v_{0,1}, v_{1,\overline{1}}$ collectively donate a charge of at least $\frac{1}{6} + \frac{1}{15} + \frac{1}{15} = \frac{9}{30} > \frac{1}{4}$. 

Overall, $v_{1,0}, v_{0,1}, v_{1,\overline{1}}$ collectively donate a charge of at least $\frac{1}{4}$ to $v_{0,0}$ according to (R1) and (R2). 
Similarly, $v_{\overline{1},0}, v_{\overline{1},1}, v_{0,\overline{1}}$ collectively donate a charge of at least $\frac{1}{4}$ to $v_{0,0}$ according to (R1) and (R2). 
This implies, $ch_2(v_{0,0}) \geq \frac{1}{4} +\frac{1}{4} = \frac{1}{2}$. 

\medskip

\noindent\textbf{Case 3: $\mathbf{|I(v_{0,0})| \leq 1}$.}
Note that, $I(v_{0,0})$ must be non-empty since $C$ is a dominating set in particular. Thus, without loss of generality, let $v_{1,0}$ be the only codeword neighbor of $v_{0,0}$. 
Observe that,  $I(\{v_{0,0}, v_{0,1}\}) = I(v_{0,1})$, which is not possible. \qed

\subsection{Proof of Theorem~\ref{thm ldb}(c) (lower bound)} \label{proof ldbc}
In this section, we show that $D^B_{1,3}\geq\frac{4}{7}$.
It is enough to show that 
     $ch_2(v_{0,0}) \geq \frac{4}{3} = (\frac{4}{7})/(1-\frac{4}{7})$ assuming $v_{0,0}$ is a non-codeword due to Lemma~\ref{lem lower bound} along with the discussion in Section~\ref{subsection general technique}. 

       Hence let us calculate $ch_2(u)$ for  each 
       non-codeword $u$  based on how many adjacent codewords they have, and their positions. 
       For this part, without loss of generality, we will assume $v_{0,0}$ to be a non-codeword, and consider all the cases under this assumption. That is, under different scenarios, we want to show that $ch_2(v_{0,0}) \geq \frac{4}{3}$, which will conclude the proof. 

\begin{figure}[b]
             \centering
             \begin{subfigure}[b]{0.5\textwidth}
                \centering
                \scalebox{.8}{\begin{tikzpicture}
	\begin{pgfonlayer}{nodelayer}
		\node [style=Empty] (39) at (-6.5, 3) {};
		\node [style=Empty] (40) at (-5.5, 3) {};
		\node [style=Empty] (41) at (-4.5, 3) {};
		\node [style=Empty] (42) at (-3.5, 3) {};
		\node [style=Vertex] (52) at (-6, 4) {$v_{0,0}$};
		\node [style=Vertex] (53) at (-5, 4) {$v_{1,0}$};
		\node [style=Vertex] (54) at (-4, 4) {$v_{2,0}$};
		\node [style=Empty] (64) at (-6.5, 5) {};
		\node [style=Empty] (65) at (-5.5, 5) {};
		\node [style=Empty] (66) at (-4.5, 5) {};
		\node [style=Empty] (67) at (-3.5, 5) {};
	\end{pgfonlayer}
	\begin{pgfonlayer}{edgelayer}
		\draw (39) to (52);
		\draw (52) to (65);
		\draw (66) to (53);
		\draw (53) to (40);
		\draw (42) to (41);
		\draw (41) to (40);
		\draw (40) to (39);
		\draw (52) to (53);
		\draw (54) to (53);
		\draw (67) to (66);
		\draw (66) to (65);
		\draw (65) to (64);
		\draw (64) to (52);
		\draw (52) to (40);
		\draw (41) to (54);
		\draw (54) to (67);
		\draw (42) to (54);
		\draw (54) to (66);
		\draw (41) to (53);
		\draw (53) to (65);
	\end{pgfonlayer}
\end{tikzpicture}}
                \caption{Three non-codewords in a straight line.\\$ $}
                \label{Fig:ForbidConsecutive-appendix}
             \end{subfigure}
             \hfill
             \begin{subfigure}[b]{0.4\textwidth}
                \centering
                \scalebox{.8}{\begin{tikzpicture}
	\begin{pgfonlayer}{nodelayer}
		\node [style=Empty] (27) at (-6, 2) {};
		\node [style=Vertex] (28) at (-5, 2) {$v_{0,\overline{1}}$};
		\node [style=Empty] (29) at (-4, 2) {};
		\node [style=Empty] (30) at (-3, 2) {};
		\node [style=Empty] (39) at (-6.5, 3) {};
		\node [style=Empty] (40) at (-5.5, 3) {};
		\node [style=Vertex] (41) at (-4.5, 3) {$v_{0,0}$};
		\node [style=Vertex] (42) at (-3.5, 3) {$v_{1,0}$};
		\node [style=Vertex] (52) at (-6, 4) {$v_{\overline{2},1}$};
		\node [style=Vertex] (53) at (-5, 4) {\footnotesize$v_{\overline{1},1}$};
		\node [style=Empty] (54) at (-4, 4) {};
		\node [style=Empty] (55) at (-3, 4) {};
		\node [style=Empty] (64) at (-6.5, 5) {};
		\node [style=Vertex] (65) at (-5.5, 5) {\footnotesize$v_{\overline{2},2}$};
		\node [style=Vertex] (66) at (-4.5, 5) {$v_{\overline{1},2}$};
		\node [style=Empty] (67) at (-3.5, 5) {};
	\end{pgfonlayer}
	\begin{pgfonlayer}{edgelayer}
		\draw (39) to (52);
		\draw (52) to (65);
		\draw (66) to (53);
		\draw (53) to (40);
		\draw (40) to (27);
		\draw (27) to (28);
		\draw (29) to (28);
		\draw (29) to (30);
		\draw (42) to (41);
		\draw (41) to (40);
		\draw (40) to (39);
		\draw [in=180, out=0] (52) to (53);
		\draw (54) to (53);
		\draw (54) to (55);
		\draw (67) to (66);
		\draw (66) to (65);
		\draw (65) to (64);
		\draw (55) to (42);
		\draw (42) to (29);
		\draw (27) to (39);
		\draw (64) to (52);
		\draw (52) to (40);
		\draw (40) to (28);
		\draw (55) to (67);
		\draw (28) to (41);
		\draw (41) to (54);
		\draw (54) to (67);
		\draw (30) to (42);
		\draw (42) to (54);
		\draw (54) to (66);
		\draw (41) to (29);
		\draw (41) to (53);
		\draw (53) to (65);
	\end{pgfonlayer}
\end{tikzpicture}}
                \caption{Three non-codewords in a curve.}
                \label{Fig:ForbidCorner-appendix}
             \end{subfigure}
             \caption{Reference figures for the proof of Claim 2. }
            \label{Fig:claim}
        \end{figure}
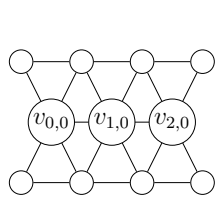
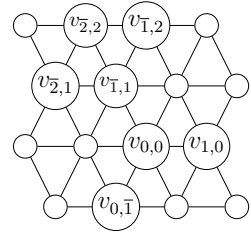

\medskip
        

        \noindent\textit{Claim 2.}\quad
            Let $C$ be a $(1,\leq \ell)$-LDB code, $\ell\geq3$, in the infinite triangular grid $\mathbf{G_6}$, then the following hold: 
            \begin{enumerate}[(a)]
                \item  There are no three distinct non-codewords $u,v,w$ such that $N[w]\subseteq N[v]\cup N[u]$.

                \item There cannot exist three non-codewords, such as $v_{0,0},v_{1,0},v_{2,0}$, in a straight line as in Figure~\ref{Fig:claim}(a).

                \item When we have the three non-codewords, such as $v_{0,\overline{1}},v_{0,0},v_{1,0}$, in a curve as in Figure \ref{Fig:claim}(b), then vertices $v_{\overline{1},1},v_{\overline{2},1},v_{\overline{1},2}$ are codewords. Moreover, if $\ell \geq 4$, then also vertex $v_{\overline{2},2}$ is a codeword. 
            \end{enumerate}
            
        \noindent\textbf{\textit{Proof of Claim:}} 
           (a) If the three non-codewords $u,v,w$ such that $N[w]\subseteq N[v]\cup N[u]$ exist, then $I(u,v)=I(w,v,u)$, a contradiction. 

\medskip
            
          (b) When we have the three non-codewords in a straight line as in Figure \ref{Fig:claim}(a), we will have $N[v_{1,0}]\subseteq N[v_{0,0}]\cup N[v_{2,0}]$ which is not possible due to the part (a) of the claim. 

          \medskip 
          
          (c) When we have the three non-codewords in a curve as in Figure \ref{Fig:claim}(b), 
          the vertices $v_{\overline{2},1}$, $v_{\overline{1},1}$, and $v_{\overline{1},2}$ are codewords, 
          as
          otherwise 
          $I(v_{0,0},v_{\overline{2},1},v_{1,0})=I(v_{\overline{2},1},v_{1,0})$,
          $I(v_{0,0},v_{\overline{1},1},v_{1,0})=I(v_{\overline{1},1},v_{1,0})$,
          or
          $I(v_{0,0},v_{\overline{1},2},v_{0,\overline{1}})=I(v_{\overline{1},2},v_{0,\overline{1}})$ 
          implying $C$ is not a $(1,\leq 3)$-LDB code. 
          Similarly, for $\ell \geq 4$, the vertex $v_{\overline{2},2}$ is a codeword, as otherwise 
          $I(v_{0,0},v_{0,\overline{1}}, v_{1,0}, v_{\overline{2},2}) = I(v_{0,\overline{1}}, v_{1,0}, v_{\overline{2},2})$ 
          implying $C$ cannot be a $(1,\leq \ell)$-LDB code for $\ell \geq 4$.
\hfill $\circ$
            
            \medskip
            
        
        
        

\medskip

        \begin{figure}[t]
            \centering
            \scalebox{.7}{\input{3LDB5Neighbours}}
            \caption{Reference figures for Cases 2 and 3 in the proof of Theorem~\ref{thm ldb}(c), where gray vertices are codewords.}
            \label{Fig:3LDB5Neighbours}
        \end{figure}

\medskip

         \noindent \textbf{Case 1: $\mathbf{|I(v_{0,0})| = 6}$.} In this case, all six neighbors of $v_{0,0}$ are codewords. Since each codeword neighbor $c$ of $v_{0,0}$ is adjacent to two other neighbors of $v_{0,0}$, it must have $|I(c)| \geq 3$. 
            Thus,  $v_{0,0}$ receives at least $\frac{1}{4}$ charge from each of its six codeword neighbors due to (R1) and its charge gets updated as $ch_1(v_{0,0}) \geq \frac{3}{2} > \frac{4}{3}$. Therefore, in particular, $ch_2(v_{0,0}) \geq \frac{4}{3}$. 

\medskip

        \noindent \textbf{Case 2: $\mathbf{|I(v_{0,0})| = 5}$.}  Let $v_{0,0}$ have exactly $5$ codeword neighbors as depicted in Figure~\ref{Fig:3LDB5Neighbours}(a). Notice that this is a general case since $\mathbf{G_6}$ is edge-transitive. Accordingly, we have assumed $v_{1,0}$ to be the only  non-codeword neighbor of $v_{0,0}$. 
        
        Observe that, either of $v_{1,1}$ or $v_{2,\overline{1}}$ must be a codeword, or else 
            we will have $I(\{v_{1,0}, v_{1,1}, v_{2, \overline{1}}\}) = I(\{v_{1,1}, v_{2, \overline{1}}\})$, a contradiction. Let us, thus, assume that $v_{1,1}$ is a codeword. This assumption is without loss of generality due to the reflectional symmetry (along $X$-axis) of the graph, and due  to the positioning of the relevant codewords with respect to $v_{0,0}$. 
            
            Furthermore, consider the following four sets of vertices: 
            $S_1 = \{ v_{0,2}, v_{\overline{1},2}, v_{\overline{2},2} \}$, 
            $S_2 = \{ v_{\overline{2},2},$ 
            $v_{\overline{2},1}, v_{\overline{2},0} \}$, 
            $S_3 = \{ v_{\overline{2},0}, v_{\overline{1},\overline{1}}, v_{0,\overline{2}} \}$ and 
            $S_4 = \{ v_{0,\overline{2}}, v_{1,\overline{2}}, v_{2,\overline{2}} \}$.  
            Note that each of the above-listed sets of vertices
            must contain at least one codeword to avoid 
            three non-codewords in a straight line due to Claim 2(b). 
            Thus, observe that, we must have at least two such vertices from $\cup_{i = 1}^{4} S_i$ as codewords, 
            each of which is adjacent to at least one codeword neighbor of $v_{0,0}$ among $N(v_{0,0}) \setminus \{v_{1, \overline{1}}\}$. 
            
            Therefore, $v_{0,0}$ has exactly $5$ codeword neighbors and among them at least two are adjacent to at least $3$ codewords, two of them are adjacent to at least $2$ codewords, and one of them is adjacent to one or more codewords. Hence, after applying Rule (R1), we have $ch_1(v_{0,0}) \geq \frac{1}{5} + \frac{2}{4} + \frac{2}{3} = \frac{41}{30} > \frac{4}{3}$. 
            Therefore, in particular, $ch_2(v_{0,0}) \geq \frac{4}{3}$.

        \medskip
            
        \noindent \textbf{Case 3: $\mathbf{|I(v_{0,0})| = 4}$.}  Since three non-codewords cannot be in a straight line
        due to Claim 2(b), we need to consider two valid subcases here based on the adjacency of the two non-codeword neighbors of $v_{0,0}$. 

\medskip
        
        \noindent \textbf{Subcase 1: the two non-codeword neighbors of $\mathbf{v_{0,0}}$ are non-adjacent.} In this subcase, depicted in  Figure~\ref{Fig:3LDB5Neighbours}(b), $v_{0,0}$ has exactly two non-codeword neighbors: $v_{1, 0}$ and $v_{0, \overline{1}}$. Here we  capture the case where the two non-codeword neighbors of $v_{0,0}$ are non-adjacent (and do not form a straight line along with $v_{0,0}$). Due to vertex transitivity of 
        $\mathbf{G_6}$, one of the non-codeword neighbors (for us, $v_{1,0}$) can be fixed without loss of generality. The choice of the second non-codeword neighbor can be either $v_{0, \overline{1}}$ or 
        $v_{\overline{1}, 1}$. Due to the reflectional symmetry along the $X$-axis, we can choose the former option without loss of generality.

         \sloppy Note that, 
         $v_{\overline{2},1}, v_{\overline{1},2}$ are codewords due to Claim 2(c). 
     Moreover, notice that, at least one of  
         $v_{\overline{1},\overline{1}}$
         and $v_{1,\overline{2}}$ is a codeword, as $I(\{ v_{\overline{1},\overline{1}}, v_{0,\overline{1}}, v_{1,\overline{2}} \}) = I(\{ v_{\overline{1},\overline{1}}, v_{1,\overline{2}} \})$. 
         Similarly, there is a codeword among 
         $v_{2,\overline{1}}$
         and $v_{1,1}$, as $I(\{v_{1,0}, v_{2,\overline{1}}, v_{1,1}\}) = I(\{v_{2,\overline{1}}, v_{1,1}\})$.

If  both $v_{1,1}$ and $v_{\overline{1}, \overline{1}}$ are codewords, then 
$|I(v_{0,1})| \geq 4$, 
$|I(v_{\overline{1},1})| \geq 5$,
$|I(v_{\overline{1},0})| \geq 4$, 
and
$|I(v_{1,\overline{1}})| \geq 1$. 
This will imply, $ch_1(v_{0,0}) \geq \frac{1}{3} + \frac{1}{2} + \frac{1}{3} + \frac{1}{6} = \frac{4}{3}$.

Notice that,   if  either $v_{1,1}$ or $v_{\overline{1}, \overline{1}}$ is a non-codeword, then
$v_{2,\overline{2}}$ is a codeword by Claim 2(c). In other words, 
if $v_{2,\overline{2}}$ is a non-codeword, then both $v_{1,1}$ and $v_{\overline{1}, \overline{1}}$ are codewords, a case already considered. 
So, we can suppose that $v_{2,\overline{2}}$ is a codeword.

If both $v_{1,{\overline{2}}}$ and $v_{2, \overline{1}}$ are codewords, then 
$|I(v_{0,1})| \geq 3$, 
$|I(v_{\overline{1},1})| \geq 5$,
$|I(v_{\overline{1},0})| \geq 3$, 
and
$|I(v_{1,\overline{1}})| \geq 4$. 
This will imply, $ch_1(v_{0,0}) \geq \frac{1}{4} + \frac{1}{2} + \frac{1}{4} + \frac{1}{3} = \frac{4}{3}$.

If exactly one of $v_{1,{\overline{2}}}$ and $v_{2, \overline{1}}$ is a codeword, then we can assume that $v_{1,\overline{1}}$ is a codeword, and $v_{1,1}$ is not a codeword. This assumption is without loss of generality due to the reflectional symmetry along the straight line joining the vertices $v_{0,0}$ and $v_{1, \overline{1}}$. However, in this scenario $v_{2, \overline{1}}$ must be a codeword. This implies, 
$|I(v_{0,1})| \geq 3$, 
$|I(v_{\overline{1},1})| \geq 5$,
$|I(v_{\overline{1},0})| \geq 4$, 
and
$|I(v_{1,\overline{1}})| \geq 3$. 
This will imply, $ch_1(v_{0,0}) \geq \frac{1}{4} + \frac{1}{2} + \frac{1}{3} + \frac{1}{4} = \frac{4}{3}$.

That means, in this subcase we have $ch_1(v_{0,0}) \geq \frac{4}{3}$, and thus, $ch_2(v_{0,0}) \geq \frac{4}{3}$.

\medskip 

        \noindent \textbf{Subcase 2: the two non-codeword neighbors of $\mathbf{v_{0,0}}$ are adjacent.}
        In this subcase, depicted in  Figure~\ref{Fig:3LDB5Neighbours}(c), $v_{0,0}$ has exactly two non-codeword neighbors: $v_{0, 1}$ and $v_{1, 0}$.         
        Here we capture the case where the two non-codeword neighbors of $v_{0,0}$ are adjacent. 
        Since $\mathbf{G_6}$ is also triangle-transitive, 
        we can choose any 
        two adjacent neighbors of $v_{0,0}$ 
        without loss of generality.

        \sloppy Notice that, 
        $v_{\overline{1},2}, v_{0,2},v_{2,0},v_{2,\overline{1}}$ are all codewords to avoid three non-codewords in straight line due to Claim 2(b).  
        Furthermore, one among $v_{\overline{2}, 1}$ and $v_{1, \overline{2}}$ must be a codeword, as $I(\{ v_{\overline{2}, 1}, v_{1, \overline{2}}, v_{0,0} \}) = I(\{ v_{\overline{2}, 1}, v_{1, \overline{2}} \})$. Let $L$ denote the line through the points $(1,1)$ and $(0,0)$ in $\mathbb{R}^2$. Due to the reflectional symmetry along $L$ of the embedding of $\mathbf{G_6}$ in $\mathbb{R}^2$, without loss of generality, we may assume that $v_{\overline{2}, 1}$ is a codeword. 

        Suppose $v_{1,1}$ is a non-codeword. Then  
        $v_{\overline{2},2}$ must be a codeword as $v_{0,0}, v_{0,1}$ and $v_{1,1}$ are in a curve, due to Claim 2(c). Similarly, 
        $v_{2, \overline{2}}$ must be a codeword as $v_{0,0}, v_{1,0}$ and $v_{1,1}$ are in a curve, due to Claim 2(c). 
        In this scenario, 
        $|I(v_{\overline{1},1})| \geq 5$, 
        $|I(v_{\overline{1},0})| \geq 4$,
        $|I(v_{0,\overline{1}})| \geq 3$,
        $|I(v_{1,\overline{1}})| \geq 4$.
       Hence, 
       $ch_1(v_{0,0}) \geq \frac{1}{2} + \frac{1}{3} + \frac{1}{4} + \frac{1}{3} = \frac{17}{12} > \frac{4}{3}$. 
       That means, if $v_{1,1}$ is a non-codeword, then $ch_2(v_{0,0}) \geq \frac{4}{3}$. 

       Thus, henceforth for this case, we may assume that 
$v_{1,1}$ is a codeword. 
Let us next consider the following sets of vertices:
$S_1 = \{ v_{\overline{2},0}, v_{\overline{1},\overline{1}}, v_{0,\overline{2}} \}$ and $S_2 = \{ v_{0,\overline{2}}, v_{1,\overline{2}}, v_{2,\overline{2}} \}$. 
Note that both of the above-listed sets of vertices must contain at least one codeword to avoid three non-codewords
in a straight line due to Claim 2(b). Also notice that, if at least two among the five vertices from $S_1 \cup S_2$ are codewords, then 
it will ensure $ch_1(v_{0,0}) \geq \frac{4}{3}$, which in turn will imply, $ch_2(v_{0,0}) \geq \frac{4}{3}$. 
If that does not happen, that is, if at least two 
among the five vertices from $S_1 \cup S_2$ are not codewords, then we must have the following situation: 
$v_{0,\overline{2}}$ is a codeword and 
$v_{\overline{2},0}, v_{\overline{1},\overline{1}}, v_{1,\overline{2}}, v_{2,\overline{2}}$ are non-codewords. Next, let us next analyze this situation.

Note that,  $v_{1,2}$ must be a codeword, as 
$I(\{v_{\overline{2},2}, v_{0,1}, v_{1,2}\}) = I(\{v_{\overline{2},2}, v_{1,2}\})$. Similarly, $v_{2,1}$ must be a codeword, as 
$I(\{v_{2,1},v_{1,0}, v_{2,\overline{2}}\}) = I(\{v_{2,1}, v_{2,\overline{2}}\})$. 
Also, $v_{3, \overline{2}}$ is a codeword 
to avoid three non-codewords $v_{1,\overline{2}}$, 
$v_{2, \overline{2}}$, $v_{3,\overline{2}}$ in a straight line due to Claim 2(b). 

Let us now calculate the updated charges of $v_{0,1}$ and $v_{1,0}$ after applying (R1). 
Observe that, $|I(v_{\overline{1},1})| \geq 4,
|I(v_{\overline{1},2})| \geq 3,
|I(v_{0,2})| \geq 4,
|I(v_{1,1})| \geq 5,
|I(v_{2,0})| \geq 4,
|I(v_{2,\overline{1}})| \geq 4,
|I(v_{1,\overline{1}})| \geq 3
$.
Hence, $ch_1(v_{0,1}) \geq \frac{1}{3} + \frac{1}{4} + \frac{1}{3} + \frac{1}{2} = \frac{17}{12} = \frac{4}{3} + \frac{1}{12}$ 
and similarly, 
$ch_1(v_{1,0}) \geq \frac{1}{2} + \frac{1}{3} + \frac{1}{3} + \frac{1}{4} = \frac{17}{12} = \frac{4}{3} + \frac{1}{12}$.
Therefore, each of $v_{0,1}$ and $v_{1,0}$ will donate $\frac{1}{12}$ amount of charge to $v_{0,0}$ while applying (R2).

Now let us calculate the updated charge of $v_{0,0}$ after applying (R1) and (R2).
Observe that, $|I(v_{\overline{1},1})| \geq 4,
|I(v_{\overline{1},0})| \geq 4,
|I(v_{0,\overline{1}})| \geq 4,
|I(v_{1,\overline{1}})| \geq 3
$.
Hence, $ch_2(v_{0,0}) \geq \frac{1}{3} + \frac{1}{3} + \frac{1}{3} + \frac{1}{4} + \frac{1}{12} +  \frac{1}{12} = \frac{17}{12}$. 

\medskip

\begin{figure}[t]
                \centering
                \scalebox{.8}{\input{3LDB3Neighbours}}
                \caption{Reference figures for Case 4 in the proof of Theorem~\ref{thm ldb}(c), where gray vertices are codewords.}
                \label{3LDB3Neighbours}
            \end{figure}

        \noindent \textbf{Case 4: $\mathbf{|I(v_{0,0})| = 3}$.}  
        Since three non-codewords cannot be in a straight line
        due to Claim 2(b), without loss of generality there may be two choices of codeword neighbors of $v_{0,0}$. 

        The first choice is depicted in Figure~\ref{3LDB3Neighbours}(a)  
        where $v_{0,0}$ has exactly three codeword neighbors $v_{1,0}, v_{\overline{1},1}, v_{0,\overline{1}} $ and exactly three non-codeword neighbors $v_{0,1},v_{\overline{1},0}, v_{1,\overline{1}}$. 
        This
        is not possible as $I(\{ v_{0,0},v_{0,1},v_{\overline{1},0} \}) = I(\{ v_{0,1},v_{\overline{1},0} \})$.

        \medskip

         \sloppy The second choice is depicted in Figure~\ref{3LDB3Neighbours}(b) where $v_{0,0}$ has exactly three codeword neighbors $v_{1,0}, v_{1,\overline{1}}, v_{0,\overline{1}}$ and
         three non-codeword neighbors $v_{\overline{1},0}, v_{\overline{1},1}$ and $v_{0,1}$.
         Observe that $v_{\overline{1}, \overline{1}}, v_{\overline{1},2}, v_{\overline{2},1}, v_{1,1}, v_{\overline{2},2}, v_{\overline{2},0},v_{0,2}$ are codewords since there cannot be three consecutive non-codewords in a straight line, due to Claim 2(b). 
        Furthermore, $v_{\overline{3},2}, v_{\overline{2},3}$ (resp., $ v_{1,\overline{2}}, v_{2,\overline{1}}$) must be codewords as $v_{\overline{1},0}, v_{\overline{1},1}, v_{0,1}$ (resp., $v_{\overline{1},0}, v_{0,0}, v_{0,1}$) are in a curve, due to Claim 2(c).

        Notice that if either $v_{2,0}$ or $v_{0,\overline{2}}$ is a codeword, then the updated charge of $v_{0,0}$ after (R1) is ensured to be $ch_1(v_{0,0}) \geq \frac{4}{3}$. 
        
        Therefore, let us assume that $v_{2,0}$ and $v_{0,\overline{2}}$ are non-codewords. 
        Observe that $I(\{ v_{\overline{3},1},$ $ v_{\overline{1},0}, v_{0,\overline{2}} \}) = I(\{ v_{\overline{3},1}, v_{0,\overline{2}} \})$, thus $v_{\overline{3},1}$ must be a codeword. 
        Similarly, $v_{\overline{1},3} $ must be a codeword.
        
        Let us calculate the updated charges of $v_{\overline{1},0}, v_{\overline{1},1}, v_{0,1}$ after (R1).
        Observe that 
        $   I(v_{0, \overline{1}}) \geq 4,
            I(v_{\overline{1}, \overline{1}}) \geq 3,
            I(v_{\overline{2},0}) \geq 4,
            I(v_{\overline{2},1}) \geq 5,
            I(v_{\overline{2},2}) \geq 5,
            I(v_{\overline{1},2}) \geq 5,
            I(v_{0,2}) \geq 4,
            I(v_{1,1}) \geq 3,            
            I(v_{1,0}) \geq 4.
        $
        Hence, $ch_1(v_{\overline{1},0}) \geq \frac{1}{3} + \frac{1}{4} + \frac{1}{3} + \frac{1}{2} = \frac{17}{12} = \frac{4}{3} + \frac{1}{12}$,   
        $ch_1(v_{\overline{1},1}) \geq \frac{1}{2} + \frac{1}{2} + \frac{1}{2} = \frac{3}{2} = \frac{4}{3} + \frac{1}{6}$ and
        $ch_1(v_{0,1}) \geq \frac{1}{3} + \frac{1}{4} + \frac{1}{3} + \frac{1}{2} = \frac{17}{12} = \frac{4}{3} + \frac{1}{12}$.
        Therefore, each of $v_{\overline{1},0}$ and $ v_{0,1}$ will donate $\frac{1}{12}$ amount of charge to $v_{0,0}$ and $v_{\overline{1},1}$ will donate $\frac{1}{6}$ amount of charge to $v_{0,0}$.

        Now let us calculate the updated charge of $v_{0,0}$ after applying (R1) and (R2).
        Observe that 
        $
            I(v_{0,\overline{1}}) \geq 4,
            I(v_{1,\overline{1}}) \geq 5,
            I(v_{1,0}) \geq 4.
        $
        Therefore, $ch_2(v_{0,0}) \geq \frac{1}{3} + \frac{1}{2} + \frac{1}{3} + \frac{1}{12} + \frac{1}{12} + \frac{1}{6} = \frac{3}{2} > \frac{4}{3}$.

        \medskip

        \noindent\textbf{Case 5: $\mathbf{|I(v_{0,0})| \leq 2}$.}
        Observe that, regardless of how the codeword 
        neighbors (at most $2$) of $v_{0,0}$ are positioned, there 
        will exist three non-codeword vertices on a straight line, contradicting Claim 2(b). Thus, this particular case never occurs. \qed



\subsection{Proof of Theorem~\ref{thm ldb}(d) (lower bound)}\label{prf ldbd}
    In this section, we show that $D^B_{1,\ell}\geq\frac{2}{3}$ for each $\ell\geq4$.
    It is enough to show that 
     $ch_2(v) \geq 2=(\frac{2}{3})/(1-\frac{2}{3})$ for each non-codeword $v$ due to Lemma~\ref{lem lower bound} along with the discussion in Section~\ref{subsection general technique}. 

 Hence, let us calculate $ch_2(u)$ for  each 
       non-codeword $u$  based on how many adjacent codewords they have, and their positions. 
       For this part, without loss of generality, we will assume $v_{0,0}$ to be a non-codeword, and consider all the cases under this assumption. That is, under different scenarios, we want to show that $ch_2(v_{0,0}) \geq 2$, which will conclude the proof.

    Notice that, any  $(1,\leq 4)$-LDB code is, in particular, a $(1,\leq 3)$-LDB code.
    Thus,  Claim 2 is  valid also for $C$.

    \medskip

        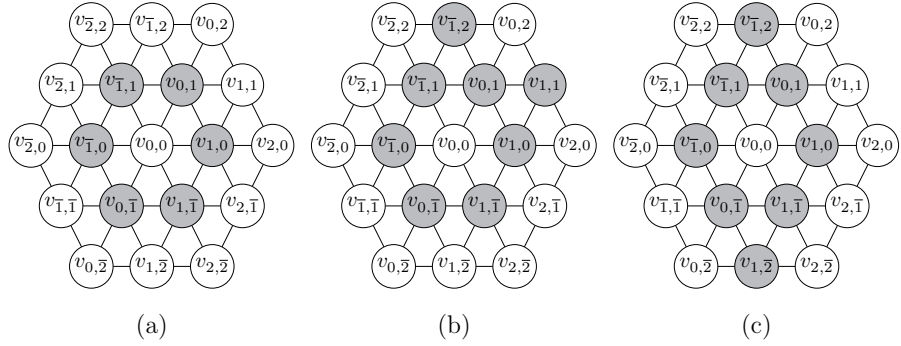
\begin{figure}[t]
            \centering
            \scalebox{.8}{\begin{tikzpicture}
	\begin{pgfonlayer}{nodelayer}
		\node [style=Vertex] (1) at (-7, 0) {\footnotesize$v_{0,\overline{2}}$};
		\node [style=Vertex] (2) at (-6, 0) {\footnotesize$v_{1,\overline{2}}$};
		\node [style=Vertex] (3) at (-5, 0) {\footnotesize$v_{2,\overline{2}}$};
		\node [style=Vertex] (13) at (-7.5, 1) {\footnotesize$v_{\overline{1},\overline{1}}$};
		\node [style=Vertex,fill=Gray!60] (14) at (-6.5, 1) {\footnotesize$v_{0,\overline{1}}$};
		\node [style=Vertex,fill=Gray!60] (15) at (-5.5, 1) {\footnotesize$v_{1,\overline{1}}$};
		\node [style=Vertex] (16) at (-4.5, 1) {\footnotesize$v_{2,\overline{1}}$};
		\node [style=Vertex] (25) at (-8, 2) {\footnotesize$v_{\overline{2},0}$};
		\node [style=Vertex,fill=Gray!60] (26) at (-7, 2) {\footnotesize$v_{\overline{1},0}$};
		\node [style=Vertex] (27) at (-6, 2) {\footnotesize$v_{0,0}$};
		\node [style=Vertex,fill=Gray!60] (28) at (-5, 2) {\footnotesize$v_{1,0}$};
		\node [style=Vertex] (29) at (-4, 2) {\footnotesize$v_{2,0}$};
		\node [style=Vertex] (38) at (-7.5, 3) {\footnotesize$v_{\overline{2},1}$};
		\node [style=Vertex,fill=Gray!60] (39) at (-6.5, 3) {\footnotesize$v_{\overline{1},1}$};
		\node [style=Vertex,fill=Gray!60] (40) at (-5.5, 3) {\footnotesize$v_{0,1}$};
		\node [style=Vertex,fill=Gray!60] (41) at (-4.5, 3) {\footnotesize$v_{1,1}$};
		\node [style=Vertex] (42) at (-7, 4) {\footnotesize$v_{\overline{2},2}$};
		\node [style=Vertex,fill=Gray!60] (43) at (-6, 4) {\footnotesize$v_{\overline{1},2}$};
		\node [style=Vertex] (44) at (-5, 4) {\footnotesize$v_{0,2}$};
		\node [style=Vertex] (46) at (-2, 0) {\footnotesize$v_{0,\overline{2}}$};
		\node [style=Vertex,fill=Gray!60] (47) at (-1, 0) {\footnotesize$v_{1,\overline{2}}$};
		\node [style=Vertex] (48) at (0, 0) {\footnotesize$v_{2,\overline{2}}$};
		\node [style=Vertex] (50) at (-2.5, 1) {\footnotesize$v_{\overline{1},\overline{1}}$};
		\node [style=Vertex,fill=Gray!60] (51) at (-1.5, 1) {\footnotesize$v_{0,\overline{1}}$};
		\node [style=Vertex,fill=Gray!60] (52) at (-0.5, 1) {\footnotesize$v_{1,\overline{1}}$};
		\node [style=Vertex] (53) at (0.5, 1) {\footnotesize$v_{2,\overline{1}}$};

		\node [style=Vertex] (54) at (-3, 2) {\footnotesize$v_{\overline{2},0}$};
		\node [style=Vertex,fill=Gray!60] (55) at (-2, 2) {\footnotesize$v_{\overline{1},0}$};
		\node [style=Vertex] (56) at (-1, 2) {\footnotesize$v_{0,0}$};
		\node [style=Vertex,fill=Gray!60] (57) at (0, 2) {\footnotesize$v_{1,0}$};
		\node [style=Vertex] (58) at (1, 2) {\footnotesize$v_{2,0}$};
		\node [style=Vertex] (59) at (-2.5, 3) {\footnotesize$v_{\overline{2},1}$};
		\node [style=Vertex,fill=Gray!60] (60) at (-1.5, 3) {\footnotesize$v_{\overline{1},1}$};
		\node [style=Vertex,fill=Gray!60] (61) at (-0.5, 3) {\footnotesize$v_{0,1}$};
		\node [style=Vertex] (62) at (0.5, 3) {\footnotesize$v_{1,1}$};
		\node [style=Vertex] (63) at (-2, 4) {\footnotesize$v_{\overline{2},2}$};
		\node [style=Vertex,fill=Gray!60] (64) at (-1, 4) {\footnotesize$v_{\overline{1},2}$};
		\node [style=Vertex] (65) at (0, 4) {\footnotesize$v_{0,2}$};
        
		\node [style=Vertex] (66) at (-12, 0) {\footnotesize$v_{0,\overline{2}}$};
		\node [style=Vertex] (67) at (-11, 0) {\footnotesize$v_{1,\overline{2}}$};
		\node [style=Vertex] (68) at (-10, 0) {\footnotesize$v_{2,\overline{2}}$};
		\node [style=Vertex] (69) at (-12.5, 1) {\footnotesize$v_{\overline{1},\overline{1}}$};
		\node [style=Vertex,fill=Gray!60] (70) at (-11.5, 1) {\footnotesize$v_{0,\overline{1}}$};
		\node [style=Vertex,fill=Gray!60] (71) at (-10.5, 1) {\footnotesize$v_{1,\overline{1}}$};
		\node [style=Vertex] (72) at (-9.5, 1) {\footnotesize$v_{2,\overline{1}}$};
		\node [style=Vertex] (73) at (-13, 2) {\footnotesize$v_{\overline{2},0}$};
		\node [style=Vertex,fill=Gray!60] (74) at (-12, 2) {\footnotesize$v_{\overline{1},0}$};
		\node [style=Vertex] (75) at (-11, 2) {\footnotesize$v_{0,0}$};
		\node [style=Vertex,fill=Gray!60] (76) at (-10, 2) {\footnotesize$v_{1,0}$};
		\node [style=Vertex] (77) at (-9, 2) {\footnotesize$v_{2,0}$};
		\node [style=Vertex] (78) at (-12.5, 3) {\footnotesize$v_{\overline{2},1}$};
		\node [style=Vertex,fill=Gray!60] (79) at (-11.5, 3) {\footnotesize$v_{\overline{1},1}$};
		\node [style=Vertex,fill=Gray!60] (80) at (-10.5, 3) {\footnotesize$v_{0,1}$};
		\node [style=Vertex] (81) at (-9.5, 3) {\footnotesize$v_{1,1}$};
		\node [style=Vertex] (82) at (-12, 4) {\footnotesize$v_{\overline{2},2}$};
		\node [style=Vertex] (83) at (-11, 4) {\footnotesize$v_{\overline{1},2}$};
		\node [style=Vertex] (84) at (-10, 4) {\footnotesize$v_{0,2}$};

        \node at (-11,-1) {(a)};
        \node at (-6,-1) {(b)};
        \node at (-1,-1) {(c)};
        
	\end{pgfonlayer}
	\begin{pgfonlayer}{edgelayer}
		\draw (13) to (26);
		\draw (26) to (39);
		\draw (40) to (27);
		\draw (27) to (14);
		\draw (14) to (1);
		\draw (1) to (2);
		\draw (2) to (3);
		\draw (16) to (15);
		\draw (15) to (14);
		\draw (13) to (14);
		\draw (13) to (25);
		\draw (25) to (26);
		\draw (26) to (27);
		\draw (27) to (28);
		\draw (29) to (28);
		\draw (41) to (40);
		\draw (40) to (39);
		\draw (39) to (38);
		\draw (29) to (16);
		\draw (16) to (3);
		\draw (3) to (15);
		\draw (15) to (27);
		\draw (27) to (39);
		\draw (40) to (28);
		\draw (28) to (16);
		\draw (2) to (15);
		\draw (15) to (28);
		\draw (28) to (41);
		\draw (41) to (29);
		\draw (2) to (14);
		\draw (14) to (26);
		\draw (26) to (38);
		\draw (38) to (25);
		\draw (13) to (1);
		\draw (38) to (42);
		\draw (42) to (39);
		\draw (39) to (43);
		\draw (43) to (40);
		\draw (40) to (44);
		\draw (44) to (41);
		\draw (44) to (43);
		\draw (43) to (42);
		\draw (50) to (55);
		\draw (55) to (60);
		\draw (61) to (56);
		\draw (56) to (51);
		\draw (51) to (46);
		\draw (46) to (47);
		\draw (47) to (48);
		\draw (53) to (52);
		\draw (52) to (51);
		\draw (50) to (51);
		\draw (50) to (54);
		\draw (54) to (55);
		\draw (55) to (56);
		\draw (56) to (57);
		\draw (58) to (57);
		\draw (62) to (61);
		\draw (61) to (60);
		\draw (60) to (59);
		\draw (58) to (53);
		\draw (53) to (48);
		\draw (48) to (52);
		\draw (52) to (56);
		\draw (56) to (60);
		\draw (61) to (57);
		\draw (57) to (53);
		\draw (47) to (52);
		\draw (52) to (57);
		\draw (57) to (62);
		\draw (62) to (58);
		\draw (47) to (51);
		\draw (51) to (55);
		\draw (55) to (59);
		\draw (59) to (54);
		\draw (50) to (46);
		\draw (59) to (63);
		\draw (63) to (60);
		\draw (60) to (64);
		\draw (64) to (61);
		\draw (61) to (65);
		\draw (65) to (62);
		\draw (65) to (64);
		\draw (64) to (63);
		\draw (69) to (74);
		\draw (74) to (79);
		\draw (80) to (75);
		\draw (75) to (70);
		\draw (70) to (66);
		\draw (66) to (67);
		\draw (67) to (68);
		\draw (72) to (71);
		\draw (71) to (70);
		\draw (69) to (70);
		\draw (69) to (73);
		\draw (73) to (74);
		\draw (74) to (75);
		\draw (75) to (76);
		\draw (77) to (76);
		\draw (81) to (80);
		\draw (80) to (79);
		\draw (79) to (78);
		\draw (77) to (72);
		\draw (72) to (68);
		\draw (68) to (71);
		\draw (71) to (75);
		\draw (75) to (79);
		\draw (80) to (76);
		\draw (76) to (72);
		\draw (67) to (71);
		\draw (71) to (76);
		\draw (76) to (81);
		\draw (81) to (77);
		\draw (67) to (70);
		\draw (70) to (74);
		\draw (74) to (78);
		\draw (78) to (73);
		\draw (69) to (66);
		\draw (78) to (82);
		\draw (82) to (79);
		\draw (79) to (83);
		\draw (83) to (80);
		\draw (80) to (84);
		\draw (84) to (81);
		\draw (84) to (83);
		\draw (83) to (82);
	\end{pgfonlayer}
\end{tikzpicture}}
            \caption{Reference figures for Case 1 in the proof of Theorem~\ref{thm ldb}(d), where gray vertices are codewords.}
            \label{6NeighboursNonPrivate}
        \end{figure}
        
        \noindent \textbf{Case 1: $\mathbf{|I(v_{0,0})| = 6}$.}  In this case, all six neighbors of $v_{0,0}$ are codewords. 
        The situation is depicted in Figure \ref{6NeighboursNonPrivate}(a).

        Observe that, if $v_{\overline{1},2}$, $v_{2,\overline{1}}$, $v_{\overline{1},\overline{1}}$ are all 
        non-codewords, then $I(\{ v_{0,0}, v_{\overline{1},2}, v_{2,\overline{1}}, v_{\overline{1},\overline{1}} \}) = I(\{ v_{\overline{1},2}, v_{2,\overline{1}}, v_{\overline{1},\overline{1}} \})$. Thus, at least one of $v_{\overline{1},2}$, $v_{2,\overline{1}}$, $v_{\overline{1},\overline{1}}$ is a codeword.
        Hence, without loss of generality, let us assume that $v_{\overline{1},2}$ is a codeword. 

        Similarly, notice that 
        if $v_{1,1}$, $v_{\overline{2},1}$, $v_{1,\overline{2}}$ are all 
        non-codewords, then $I(\{ v_{0,0}, v_{1,1}, v_{\overline{2},1}, v_{1,\overline{2}} \}) = I(\{  v_{1,1}, v_{\overline{2},1}, v_{1,\overline{2}} \})$. Thus, at least one among $ v_{1,1}, v_{\overline{2},1}, v_{1,\overline{2}}$ is a codeword.
        Since we have already chosen 
        $v_{\overline{1},2}$ as a codeword, the choice of one codeword among 
        $v_{\overline{1},2},  v_{1,\overline{2}}, v_{1,1}$ will, without loss of generality, result in two cases: the ones depicted in Figures~\ref{6NeighboursNonPrivate}(b) and (c), respectively.

        Let us assume the scenario depicted in Figure~\ref{6NeighboursNonPrivate}(b). Observe that, 
        the three vertices $v_{\overline{2},0}, v_{\overline{1},\overline{1}}, v_{0,\overline{2}}$ are on a straight line, and thus one of them must be a codeword 
        due to Claim 2(b). 
        Hence, we  have $|I(u)| \geq 4$ for some 
        $u \in \{ v_{0,\overline{1}}, v_{\overline{1},0} \}$. This implies that 
        one of $v_{0,\overline{1}}$ and $v_{\overline{1},0}$  donates at least  
        $\frac{1}{3}$ charge and the other one  donates at least $\frac{1}{4}$ to $v_{0,0}$. 
        We already have 
        $
            |I(v_{\overline{1},1})| \geq 4,
            |I(v_{0,1})| \geq 5,
            |I(v_{1,0})| \geq 4,
            |I(v_{1,\overline{1}})| \geq 3  
        $.
        Therefore, $ch_1(v_{0,0}) \geq   \frac{1}{3} + \frac{1}{4}  + \frac{1}{3} + \frac{1}{2} + \frac{1}{3} + \frac{1}{4} = 2$.

Next let us assume the scenario depicted in Figure~\ref{6NeighboursNonPrivate}(c). 
Consider the following two sets: $S_1 = \{ v_{0,2}, v_{1,1}, v_{2,0} \}$ and $S_2 = \{ v_{\overline{2},0}, v_{\overline{1},\overline{1}}, v_{0,\overline{2}} \}$.
Note that the vertices of $S_1$ are in a straight line, and thus one of them must be a codeword. Similarly, the vertices of $S_2$ are in a straight line, and thus one of them must be a codeword. 
This will imply $ch_1(v_{0,0}) \geq 2$, irrespective of the choice of the codewords from $S_1$ and $S_2$. 

That means, $ch_1(v_{0,0}) \geq 2$ in this subcase, which in turn implies, $ch_2(v_{0,0}) \geq 2$ 
as well.

     \medskip

        \noindent \textbf{Case 2: $\mathbf{|I(v_{0,0})| = 5}$.}
        In this case $v_{0,0}$ has exactly five codeword neighbors $v_{0,1}, v_{\overline{1},1}, v_{\overline{1},0}$, $v_{0,\overline{1}}, v_{1,\overline{1}}$ and one non-codeword neighbor $v_{1,0}$.
        This case is depicted in Figure~\ref{5Neighbours5N}(a).
        Due to Claim 2(b), $v_{2,0}$ must be a codeword.
Moreover, if $v_{\overline{1},\overline{1}}$ and $v_{\overline{2},1}$ are non-codewords then $I(\{ v_{\overline{1},\overline{1}}, v_{\overline{2},1}, v_{0,0}, v_{1,0} \}) = I(\{ v_{\overline{1},\overline{1}}, v_{\overline{2},1}, v_{1,0} \})$. 
        Thus,  $v_{\overline{1},\overline{1}}$ or $v_{\overline{2},1}$ must be a codeword.
        Let us assume without loss of generality that $v_{\overline{2},1}$ is a codeword. 

        \begin{figure}[t]
            \centering
            \scalebox{.8}{\begin{tikzpicture}
	\begin{pgfonlayer}{nodelayer}
		\node [style=Vertex] (221) at (3, 0) {\footnotesize$v_{0,\overline{2}}$};
		\node [style=Vertex] (222) at (4, 0) {\footnotesize$v_{1,\overline{2}}$};
		\node [style=Empty] (223) at (5, 0) {};
		\node [style=Vertex] (224) at (2.5, 1) {\footnotesize$v_{\overline{1},\overline{1}}$};
		\node [style=Vertex,fill=Gray!60] (225) at (3.5, 1) {\footnotesize$v_{0,\overline{1}}$};
		\node [style=Vertex,fill=Gray!60] (226) at (4.5, 1) {\footnotesize$v_{1,\overline{1}}$};
		\node [style=Empty] (227) at (2, 2) {};
		\node [style=Vertex,fill=Gray!60] (228) at (3, 2) {\footnotesize$v_{\overline{1},0}$};
		\node [style=Vertex] (229) at (4, 2) {\footnotesize$v_{0,0}$};
		\node [style=Vertex] (230) at (5, 2) {\footnotesize$v_{1,0}$};
		\node [style=Vertex, fill=Gray!60] (231) at (2.5, 3) {\footnotesize$v_{\overline{2},1}$};
		\node [style=Vertex,fill=Gray!60] (232) at (3.5, 3) {\footnotesize$v_{\overline{1},1}$};
		\node [style=Vertex, fill=Gray!60] (233) at (4.5, 3) {\footnotesize$v_{0,1}$};
		\node [style=Vertex] (234) at (3, 4) {\footnotesize$v_{\overline{2},2}$};
		\node [style=Vertex] (235) at (4, 4) {\footnotesize$v_{\overline{1},2}$};
		\node [style=Empty] (236) at (5, 4) {};
		\node [style=Vertex] (241) at (5.5, 1) {\footnotesize$v_{2,\overline{1}}$};
		\node [style=Vertex] (242) at (5.5, 3) {\footnotesize$v_{1,1}$};
		\node [style=Vertex,fill=Gray!60] (243) at (6, 2) {\footnotesize$v_{2,0}$};
		\node [style=Vertex] (268) at (13, 0) {\footnotesize$v_{0,\overline{2}}$};
		\node [style=Vertex] (269) at (14, 0) {\footnotesize$v_{1,\overline{2}}$};
		\node [style=Vertex] (270) at (15, 0) {\footnotesize$v_{2,\overline{2}}$};
		\node [style=Vertex,fill=Gray!60] (271) at (12.5, 1) {\footnotesize$v_{\overline{1},\overline{1}}$};
		\node [style=Vertex,fill=Gray!60] (272) at (13.5, 1) {\footnotesize$v_{0,\overline{1}}$};
		\node [style=Vertex,fill=Gray!60] (273) at (14.5, 1) {\footnotesize$v_{1,\overline{1}}$};
		\node [style=Vertex] (274) at (12, 2) {\footnotesize$v_{\overline{2},0}$};
		\node [style=Vertex,fill=Gray!60] (275) at (13, 2) {\footnotesize$v_{\overline{1},0}$};
		\node [style=Vertex] (276) at (14, 2) {\footnotesize$v_{0,0}$};
		\node [style=Vertex] (277) at (15, 2) {\footnotesize$v_{1,0}$};
		\node [style=Vertex,fill=Gray!60] (278) at (12.5, 3) {\footnotesize$v_{\overline{2},1}$};
		\node [style=Vertex,fill=Gray!60] (279) at (13.5, 3) {\footnotesize$v_{\overline{1},1}$};
		\node [style=Vertex,fill=Gray!60] (280) at (14.5, 3) {\footnotesize$v_{0,1}$};
		\node [style=Vertex] (281) at (13, 4) {\footnotesize$v_{\overline{2},2}$};
		\node [style=Vertex,fill=Gray!60] (282) at (14, 4) {\footnotesize$v_{\overline{1},2}$};
		\node [style=Vertex] (283) at (15, 4) {\footnotesize$v_{0,2}$};
		\node [style=Vertex] (284) at (15.5, 1) {\footnotesize$v_{2,\overline{1}}$};
		\node [style=Vertex] (285) at (15.5, 3) {\footnotesize$v_{1,1}$};
		\node [style=Vertex, fill=Gray!60] (286) at (16, 2) {\footnotesize$v_{2,0}$};
		\node [style=Empty] (287) at (8, 0) {};
		\node [style=Empty] (288) at (9, 0) {};
		\node [style=Empty] (289) at (10, 0) {};
		\node [style=Vertex] (290) at (7.5, 1) {\footnotesize$v_{\overline{1},\overline{1}}$};
		\node [style=Vertex,fill=Gray!60] (291) at (8.5, 1) {\footnotesize$v_{0,\overline{1}}$};
		\node [style=Vertex,fill=Gray!60] (292) at (9.5, 1) {\footnotesize$v_{1,\overline{1}}$};
		\node [style=Empty] (293) at (7, 2) {};
		\node [style=Vertex,fill=Gray!60] (294) at (8, 2) {\footnotesize$v_{\overline{1},0}$};
		\node [style=Vertex] (295) at (9, 2) {\footnotesize$v_{0,0}$};
		\node [style=Vertex] (296) at (10, 2) {\footnotesize$v_{1,0}$};
		\node [style=Vertex,fill=Gray!60] (297) at (7.5, 3) {\footnotesize$v_{\overline{2},1}$};
		\node [style=Vertex,fill=Gray!60] (298) at (8.5, 3) {\footnotesize$v_{\overline{1},1}$};
		\node [style=Vertex,fill=Gray!60] (299) at (9.5, 3) {\footnotesize$v_{0,1}$};
		\node [style=Vertex,fill=Gray!60] (300) at (8, 4) {\footnotesize$v_{\overline{2},2}$};
		\node [style=Vertex,fill=Gray!60] (301) at (9, 4) {\footnotesize$v_{\overline{1},2}$};
		\node [style=Empty] (302) at (10, 4) {};
		\node [style=Vertex] (303) at (10.5, 1) {\footnotesize$v_{2,\overline{1}}$};
		\node [style=Vertex] (304) at (10.5, 3) {\footnotesize$v_{1,1}$};
		\node [style=Vertex,fill=Gray!60] (305) at (11, 2) {\footnotesize$v_{2,0}$};
        

        \node at (4, -1) {(a)};
        \node at (9, -1) {(b)};
        \node at (14, -1) {(c)};
        
	\end{pgfonlayer}
	\begin{pgfonlayer}{edgelayer}


		\draw (224) to (228);
		\draw (228) to (232);
		\draw (233) to (229);
		\draw (229) to (225);
		\draw (225) to (221);
		\draw (221) to (222);
		\draw (222) to (223);
		\draw (226) to (225);
		\draw (224) to (225);
		\draw (224) to (227);
		\draw (227) to (228);
		\draw (228) to (229);
		\draw (229) to (230);
		\draw (233) to (232);
		\draw (232) to (231);
		\draw (223) to (226);
		\draw (226) to (229);
		\draw (229) to (232);
		\draw (233) to (230);
		\draw (222) to (226);
		\draw (226) to (230);
		\draw (222) to (225);
		\draw (225) to (228);
		\draw (228) to (231);
		\draw (231) to (227);
		\draw (224) to (221);
		\draw (231) to (234);
		\draw (232) to (235);
		\draw (235) to (233);
		\draw (233) to (236);
		\draw (232) to (234);
		\draw (234) to (235);
		\draw (235) to (236);
		\draw (236) to (242);
		\draw (242) to (243);
		\draw (243) to (241);
		\draw (241) to (223);
		\draw (241) to (230);
		\draw [in=180, out=0] (230) to (243);
		\draw (230) to (242);
		\draw (241) to (226);
		\draw (242) to (233);
		\draw (271) to (275);
		\draw (275) to (279);
		\draw (280) to (276);
		\draw (276) to (272);
		\draw (272) to (268);
		\draw (268) to (269);
		\draw (269) to (270);
		\draw (273) to (272);
		\draw (271) to (272);
		\draw (271) to (274);
		\draw (274) to (275);
		\draw (275) to (276);
		\draw (276) to (277);
		\draw (280) to (279);
		\draw (279) to (278);
		\draw (270) to (273);
		\draw (273) to (276);
		\draw (276) to (279);
		\draw (280) to (277);
		\draw (269) to (273);
		\draw (273) to (277);
		\draw (269) to (272);
		\draw (272) to (275);
		\draw (275) to (278);
		\draw (278) to (274);
		\draw (271) to (268);
		\draw (278) to (281);
		\draw (279) to (282);
		\draw (282) to (280);
		\draw (280) to (283);
		\draw (279) to (281);
		\draw (281) to (282);
		\draw (282) to (283);
		\draw (283) to (285);
		\draw (285) to (286);
		\draw (286) to (284);
		\draw (284) to (270);
		\draw (284) to (277);
		\draw (277) to (286);
		\draw (277) to (285);
		\draw (284) to (273);
		\draw (285) to (280);
		\draw (290) to (294);
		\draw (294) to (298);
		\draw (299) to (295);
		\draw (295) to (291);
		\draw (291) to (287);
		\draw (287) to (288);
		\draw (288) to (289);
		\draw (292) to (291);
		\draw (290) to (291);
		\draw (290) to (293);
		\draw (293) to (294);
		\draw (294) to (295);
		\draw (295) to (296);
		\draw (299) to (298);
		\draw (298) to (297);
		\draw (289) to (292);
		\draw (292) to (295);
		\draw (295) to (298);
		\draw (299) to (296);
		\draw (288) to (292);
		\draw (292) to (296);
		\draw (288) to (291);
		\draw (291) to (294);
		\draw (294) to (297);
		\draw (297) to (293);
		\draw (290) to (287);
		\draw (297) to (300);
		\draw (298) to (301);
		\draw (301) to (299);
		\draw (299) to (302);
		\draw (298) to (300);
		\draw (300) to (301);
		\draw (301) to (302);
		\draw (302) to (304);
		\draw (304) to (305);
		\draw (305) to (303);
		\draw (303) to (289);
		\draw (303) to (296);
		\draw [in=180, out=0] (296) to (305);
		\draw (296) to (304);
		\draw (303) to (292);
		\draw (304) to (299);
	\end{pgfonlayer}
\end{tikzpicture}}
            \caption{Reference figure for Case 2 in the proof of Theorem~\ref{thm ldb}(d), where gray vertices are codewords.}
            \label{5Neighbours5N}
        \end{figure}

        If $v_{\overline{1}, \overline{1}}$ is a non-codeword (depicted in Figure~\ref{5Neighbours5N}(b)), then  $v_{\overline{2},2}$ must also be a codeword since, $I(\{ v_{\overline{1},\overline{1}}, v_{\overline{2},2}, v_{0,0}, v_{1,0} \}) = I(\{ v_{\overline{1},\overline{1}}, v_{\overline{2},2}, v_{1,0} \}) $.
        Similarly, $v_{\overline{1},2}$ must also be a codeword. 
        Thus, 
        $
            |I(v_{0,1})| \geq 3,
            |I(v_{\overline{1},1})| \geq 6,
            |I(v_{\overline{1},0})| \geq 4,
            |I(v_{0,\overline{1}})| \geq 3,
            |I(v_{1,\overline{1}})| \geq 2
        $.
        Hence, $ch_1(v_{0,0}) \geq \frac{1}{4} + 1 + \frac{1}{3} + \frac{1}{4} + \frac{1}{5} = \frac{61}{30} = 2 + \frac{1}{30}$. Therefore, in particular, $ch_2(v_{0,0}) \geq 2$. 
          
        So, we may now suppose that also
        $v_{\overline{1}, \overline{1}}$ is a codeword. This is depicted in Figure~\ref{5Neighbours5N}(c). Observe that, if $v_{\overline{2}, 0}$ is a codeword, then 
        $ch_1(v_{0,0}) \geq 2$. 
        Thus, we may assume that $v_{\overline{2}, 0}$ is a non-codeword. 
Next notice that $I(\{ v_{\overline{1},2}, v_{1,\overline{2}}, v_{\overline{2},0}, v_{0,0} \}) = I(\{ v_{\overline{1},2}, v_{1,\overline{2}}, v_{\overline{2},0} \})$,
which implies that at least one 
of $v_{\overline{1},2}$ and $v_{1,\overline{2}}$ is a 
codeword. Hence, without loss of generality, assume 
that $v_{\overline{1}, 2}$ is a codeword. 
    Further note that, 
    if $v_{1,\overline{2}}$ is a codeword, then 
        $ch_1(v_{0,0}) \geq 2$. 
        Thus, we may assume that
        $v_{1,\overline{2}}$ is a non-codeword. 

        Notice that, one among $\{ v_{0,2}, v_{2,\overline{1}} \}$ and one among $\{ v_{1,1}, v_{2,\overline{2}} \}$ are codewords due to Claim 2(c). This will result in four different subcases (see Figure~\ref{5Neighbours5N1} for reference). We list them below. 
 \begin{figure}[t]
            \centering
            \scalebox{.8}{\input{5Neighbours5n-2}}
            \caption{Reference figures for the subcases of Case 2 in the proof of Theorem~\ref{thm ldb}(d), where gray vertices are codewords.}
            \label{5Neighbours5N1}
        \end{figure}
        
 \begin{enumerate}[(a)]
     \item Suppose $v_{0,2}$ and $v_{2,\overline{2}}$ are codewords, see Figure~\ref{5Neighbours5N1}(a). Note that,   
     one among $\{v_{1,1}, v_{2, \overline{1}}\}$ is a codeword since 
        $I( \{  v_{1,0}, v_{0,0},v_{1,1},v_{2,\overline{1}} \} ) = I(\{v_{0,0},v_{1,1},v_{2,\overline{1}}\})$. 
    This will imply $ch_1(v_{0,0}) \geq 2$.

    \item Suppose
    $v_{2,\overline{1}}$ and $v_{1,1}$ are codewords, see Figure~\ref{5Neighbours5N1}(b). Note that,   
         one among $\{v_{0,\overline{2}}, v_{1,\overline{2}}, v_{2, \overline{2}}\}$ is a codeword. Since $v_{1,\overline{2}}$ is a non-codeword according to our assumption, either 
         $v_{0,\overline{2}}$ or 
         $ v_{2, \overline{2}}$ must be a codeword. 
    In either scenario, $ch_1(v_{0,0}) \geq 2$. 
    
    \item Suppose $v_{0,2}$ and $v_{1,1}$ are codewords, see Figure~\ref{5Neighbours5N1}(c). Note that
    one among $\{v_{0,\overline{2}}, v_{1,\overline{2}}, v_{2, \overline{2}}\}$ is a codeword. Since $v_{1,\overline{2}}$ is a non-codeword according to our assumption, either 
         $v_{0,\overline{2}}$ or 
         $ v_{2, \overline{2}}$ must be a codeword. 
    In either scenario, $ch_1(v_{0,0}) \geq 2$. 
    
    \item Suppose $v_{2,\overline{1}}$ and $v_{2,\overline{2}}$ are codewords, see Figure~\ref{5Neighbours5N1}(d). 
    Notice that $v_{1,1}$ and $v_{0,2}$ are non-codewords due to Cases 2(a) and (b)
    respectively.
    Observe that both $v_{3,\overline{1}},v_{3,\overline{2}}$ are codewords, due to Claim 2(c).
    Furthermore, $v_{3,0}$, $v_{2,1}$, $v_{2,2}$ and $v_{3,1}$ are also codewords since $v_{1,0},v_{1,1},v_{0,2}$ are in a curve, due to Claim 2(c).
    Since $v_{1,0},v_{1,1}, v_{1,2}$ are in a straight line, due to Claim 2(b), $v_{1,2}$ is a codeword.
    
    Hence, after applying (R1) we have, $ch_1(v_{1,1}) \geq \frac{1}{4} + 1 + \frac{1}{2} + \frac{1}{4} = 2$, $ch_1(v_{1,0}) \geq \frac{1}{2} + 1 + \frac{1}{3} + \frac{1}{4} = 2 + \frac{1}{12}$ and $ch_1(v_{0,0}) = \frac{1}{4} + \frac{1}{2} + \frac{1}{2} + \frac{1}{3} + \frac{1}{3} = 2 -\frac{1}{12}$.
    Thus, $v_{1,0}$ donates $\frac{1}{12}$ charge to $v_{0,0}$ according to (R2). Therefore, $ch_2(v_{0,0}) = ch_1(v_{0,0}) + \frac{1}{12} = 2$.
\end{enumerate}

\noindent Thus, $ch_2(v_{0,0}) \geq 2$ if $v_{0,0}$ has exactly five codeword neighbors.


\medskip



    \begin{figure}[t]
        \centering
        \begin{subfigure}[b]{0.49\textwidth}
            \centering
            \scalebox{0.8}{\begin{tikzpicture}
	\begin{pgfonlayer}{nodelayer}
		
\node[style=Vertex] (v-3-3) at (-3,-3) {\footnotesize$v_{\overline{2},\overline{3}}$};
\node[style=Vertex] (v-2-3) at (-2,-3) {\footnotesize$v_{\overline{1},\overline{3}}$};
\node[style=Vertex] (v-1-3) at (-1,-3) {\footnotesize$v_{0,\overline{3}}$};
\node[style=Vertex] (v0-3)  at ( 0,-3) {\footnotesize$v_{1,\overline{3}}$};
\node[style=Vertex] (v1-3)  at ( 1,-3) {\footnotesize$v_{2,\overline{3}}$};
\node[style=Vertex] (v2-3)  at ( 2,-3) {\footnotesize$v_{3,\overline{3}}$};
\node[style=Vertex] (v3-3)  at ( 3,-3) {\footnotesize$v_{4,\overline{3}}$};

\node[style=Vertex] (v-3-2) at (-2.5,-2) {\footnotesize$v_{\overline{2},\overline{2}}$};
\node[style=Vertex] (v-2-2) at (-1.5,-2) {\footnotesize$v_{\overline{1},\overline{2}}$};
\node[style=Vertex,fill=Gray!60] (v-1-2) at (-0.5,-2) {\footnotesize$v_{0,\overline{2}}$};
\node[style=Vertex] (v0-2)  at ( 0.5,-2) {\footnotesize$v_{1,\overline{2}}$};
\node[style=Vertex] (v1-2)  at ( 1.5,-2) {\footnotesize$v_{2,\overline{2}}$};
\node[style=Vertex] (v2-2)  at ( 2.5,-2) {\footnotesize$v_{3,\overline{2}}$};
\node[style=Vertex] (v3-2)  at ( 3.5,-2) {\footnotesize$v_{4,\overline{2}}$};

\node[style=Vertex] (v-3-1) at (-3,-1) {\footnotesize$v_{\overline{3},\overline{1}}$};
\node[style=Vertex] (v-2-1) at (-2,-1) {\footnotesize$v_{\overline{2},\overline{1}}$};
\node[style=Vertex] (v-1-1) at (-1,-1) {\footnotesize$v_{\overline{1},\overline{1}}$};
\node[style=Vertex] (v0-1)  at ( 0,-1) {\footnotesize$v_{0,\overline{1}}$};
\node[style=Vertex,fill=Gray!60] (v1-1)  at ( 1,-1) {\footnotesize$v_{1,\overline{1}}$};
\node[style=Vertex] (v2-1)  at ( 2,-1) {\footnotesize$v_{2,\overline{1}}$};
\node[style=Vertex] (v3-1)  at ( 3,-1) {\footnotesize$v_{3,\overline{1}}$};

\node[style=Vertex] (v-30) at (-2.5,0) {\footnotesize$v_{\overline{3},0}$};
\node[style=Vertex] (v-20) at (-1.5,0) {\footnotesize$v_{\overline{2},0}$};
\node[style=Vertex,fill=Gray!60] (v-10) at (-0.5,0) {\footnotesize$v_{\overline{1},0}$};
\node[style=Vertex] (v00)   at ( 0.5,0) {\footnotesize$v_{0,0}$};
\node[style=Vertex] (v10)   at ( 1.5,0) {\footnotesize$v_{1,0}$};
\node[style=Vertex,fill=Gray!60] (v20)   at ( 2.5,0) {\footnotesize$v_{2,0}$};
\node[style=Vertex] (v30)   at ( 3.5,0) {\footnotesize$v_{3,0}$};

\node[style=Vertex] (v-31) at (-3,1) {\footnotesize$v_{\overline{4},1}$};
\node[style=Vertex] (v-21) at (-2,1) {\footnotesize$v_{\overline{3},1}$};
\node[style=Vertex,fill=Gray!60] (v-11) at (-1,1) {\footnotesize$v_{\overline{2},1}$};
\node[style=Vertex,fill=Gray!60] (v01)   at ( 0,1) {\footnotesize$v_{\overline{1},1}$};
\node[style=Vertex,fill=Gray!60] (v11)   at ( 1,1) {\footnotesize$v_{0,1}$};
\node[style=Vertex] (v21)   at ( 2,1) {\footnotesize$v_{1,1}$};
\node[style=Vertex] (v31)   at ( 3,1) {\footnotesize$v_{2,1}$};
\node[style=Vertex] (v100)   at ( 4,1) {\footnotesize$v_{3,1}$};

\node[style=Vertex] (v-32) at (-2.5,2) {\footnotesize$v_{\overline{4},2}$};
\node[style=Vertex] (v-22) at (-1.5,2) {\footnotesize$v_{\overline{3},2}$};
\node[style=Vertex,fill=Gray!60] (v-12) at (-0.5,2) {\footnotesize$v_{\overline{2},2}$};
\node[style=Vertex,fill=Gray!60] (v02)   at ( 0.5,2) {\footnotesize$v_{\overline{1},2}$};
\node[style=Vertex] (v12)   at ( 1.5,2) {\footnotesize$v_{0,2}$};
\node[style=Vertex] (v22)   at ( 2.5,2) {\footnotesize$v_{1,2}$};
\node[style=Vertex] (v32)   at ( 3.5,2) {\footnotesize$v_{2,2}$};

\node[style=Vertex] (v-33) at (-3,3) {\footnotesize$v_{\overline{5},3}$};
\node[style=Vertex] (v-23) at (-2,3) {\footnotesize$v_{\overline{4},3}$};
\node[style=Vertex] (v-13) at (-1,3) {\footnotesize$v_{\overline{3},3}$};
\node[style=Vertex] (v03)   at ( 0,3) {\footnotesize$v_{\overline{2},3}$};
\node[style=Vertex] (v13)   at ( 1,3) {\footnotesize$v_{\overline{1},3}$};
\node[style=Vertex] (v23)   at ( 2,3) {\footnotesize$v_{0,3}$};
\node[style=Vertex] (v33)   at ( 3,3) {\footnotesize$v_{1,3}$};
        
	\end{pgfonlayer}
    
	\begin{pgfonlayer}{edgelayer}

    \draw (v100) -- (v32);
    \draw (v100) -- (v31);
    \draw (v100) -- (v30);
		
        \foreach \j in {-3,...,3}{
            \foreach \i in {-3,...,2}{ 
                \draw (v\i\j) -- (v\the\numexpr\i+1\relax\j);
            }
        }

        \foreach \j in {-3,-1,1}{
            \foreach \i in {-3,...,3}{
                \draw (v\i\j) -- (v\i\the\numexpr\j+1\relax);
            }
            \foreach \i in {-2,...,3}{ 
                \draw (v\i\j) -- (v\the\numexpr\i-1\relax\the\numexpr\j+1\relax);
            }
        }
        
        \foreach \j in {-2,0,2}{
            \foreach \i in {-3,...,3}{
                \draw (v\i\j) -- (v\i\the\numexpr\j+1\relax);
            }
            \foreach \i in {-3,...,2}{ 
                \draw (v\i\j) -- (v\the\numexpr\i+1\relax\the\numexpr\j+1\relax);
            }
        }


	\end{pgfonlayer}
\end{tikzpicture}}            
            \caption{}
        \end{subfigure}
        \hfill
        \begin{subfigure}[b]{0.49\textwidth}
            \centering
            \scalebox{0.8}{\begin{tikzpicture}
	\begin{pgfonlayer}{nodelayer}
		
\node[style=Vertex] (v-3-3) at (-3,-3) {\footnotesize$v_{\overline{2},\overline{3}}$};
\node[style=Vertex] (v-2-3) at (-2,-3) {\footnotesize$v_{\overline{1},\overline{3}}$};
\node[style=Vertex] (v-1-3) at (-1,-3) {\footnotesize$v_{0,\overline{3}}$};
\node[style=Vertex] (v0-3)  at ( 0,-3) {\footnotesize$v_{1,\overline{3}}$};
\node[style=Vertex] (v1-3)  at ( 1,-3) {\footnotesize$v_{2,\overline{3}}$};
\node[style=Vertex] (v2-3)  at ( 2,-3) {\footnotesize$v_{3,\overline{3}}$};
\node[style=Vertex] (v3-3)  at ( 3,-3) {\footnotesize$v_{4,\overline{3}}$};

\node[style=Vertex] (v-3-2) at (-2.5,-2) {\footnotesize$v_{\overline{2},\overline{2}}$};
\node[style=Vertex] (v-2-2) at (-1.5,-2) {\footnotesize$v_{\overline{1},\overline{2}}$};
\node[style=Vertex] (v-1-2) at (-0.5,-2) {\footnotesize$v_{0,\overline{2}}$};
\node[style=Vertex] (v0-2)  at ( 0.5,-2) {\footnotesize$v_{1,\overline{2}}$};
\node[style=Vertex] (v1-2)  at ( 1.5,-2) {\footnotesize$v_{2,\overline{2}}$};
\node[style=Vertex] (v2-2)  at ( 2.5,-2) {\footnotesize$v_{3,\overline{2}}$};
\node[style=Vertex] (v3-2)  at ( 3.5,-2) {\footnotesize$v_{4,\overline{2}}$};

\node[style=Vertex] (v-3-1) at (-3,-1) {\footnotesize$v_{\overline{3},\overline{1}}$};
\node[style=Vertex] (v-2-1) at (-2,-1) {\footnotesize$v_{\overline{2},\overline{1}}$};
\node[style=Vertex,fill=Gray!60] (v-1-1) at (-1,-1) {\footnotesize$v_{\overline{1},\overline{1}}$};
\node[style=Vertex,fill=Gray!60] (v0-1)  at ( 0,-1) {\footnotesize$v_{0,\overline{1}}$};
\node[style=Vertex,fill=Gray!60] (v1-1)  at ( 1,-1) {\footnotesize$v_{1,\overline{1}}$};
\node[style=Vertex,fill=Gray!60] (v2-1)  at ( 2,-1) {\footnotesize$v_{2,\overline{1}}$};
\node[style=Vertex] (v3-1)  at ( 3,-1) {\footnotesize$v_{3,\overline{1}}$};

\node[style=Vertex] (v-30) at (-2.5,0) {\footnotesize$v_{\overline{3},0}$};
\node[style=Vertex] (v-20) at (-1.5,0) {\footnotesize$v_{\overline{2},0}$};
\node[style=Vertex,fill=Gray!60] (v-10) at (-0.5,0) {\footnotesize$v_{\overline{1},0}$};
\node[style=Vertex] (v00)   at ( 0.5,0) {\footnotesize$v_{0,0}$};
\node[style=Vertex] (v10)   at ( 1.5,0) {\footnotesize$v_{1,0}$};
\node[style=Vertex,fill=Gray!60] (v20)   at ( 2.5,0) {\footnotesize$v_{2,0}$};
\node[style=Vertex] (v30)   at ( 3.5,0) {\footnotesize$v_{3,0}$};

\node[style=Vertex] (v-31) at (-3,1) {\footnotesize$v_{\overline{4},1}$};
\node[style=Vertex] (v-21) at (-2,1) {\footnotesize$v_{\overline{3},1}$};
\node[style=Vertex,fill=Gray!60] (v-11) at (-1,1) {\footnotesize$v_{\overline{2},1}$};
\node[style=Vertex,fill=Gray!60] (v01)   at ( 0,1) {\footnotesize$v_{\overline{1},1}$};
\node[style=Vertex] (v11)   at ( 1,1) {\footnotesize$v_{0,1}$};
\node[style=Vertex] (v21)   at ( 2,1) {\footnotesize$v_{1,1}$};
\node[style=Vertex] (v31)   at ( 3,1) {\footnotesize$v_{2,1}$};

\node[style=Vertex] (v-32) at (-2.5,2) {\footnotesize$v_{\overline{4},2}$};
\node[style=Vertex] (v-22) at (-1.5,2) {\footnotesize$v_{\overline{3},2}$};
\node[style=Vertex] (v-12) at (-0.5,2) {\footnotesize$v_{\overline{2},2}$};
\node[style=Vertex,fill=Gray!60] (v02)   at ( 0.5,2) {\footnotesize$v_{\overline{1},2}$};
\node[style=Vertex,fill=Gray!60] (v12)   at ( 1.5,2) {\footnotesize$v_{0,2}$};
\node[style=Vertex] (v22)   at ( 2.5,2) {\footnotesize$v_{1,2}$};
\node[style=Vertex] (v32)   at ( 3.5,2) {\footnotesize$v_{2,2}$};

\node[style=Vertex] (v-33) at (-3,3) {\footnotesize$v_{\overline{5},3}$};
\node[style=Vertex] (v-23) at (-2,3) {\footnotesize$v_{\overline{4},3}$};
\node[style=Vertex] (v-13) at (-1,3) {\footnotesize$v_{\overline{3},3}$};
\node[style=Vertex] (v03)   at ( 0,3) {\footnotesize$v_{\overline{2},3}$};
\node[style=Vertex] (v13)   at ( 1,3) {\footnotesize$v_{\overline{1},3}$};
\node[style=Vertex] (v23)   at ( 2,3) {\footnotesize$v_{0,3}$};
\node[style=Vertex] (v33)   at ( 3,3) {\footnotesize$v_{1,3}$};
        
	\end{pgfonlayer}
    
	\begin{pgfonlayer}{edgelayer}
		
        \foreach \j in {-3,...,3}{
            \foreach \i in {-3,...,2}{ 
                \draw (v\i\j) -- (v\the\numexpr\i+1\relax\j);
            }
        }

        \foreach \j in {-3,-1,1}{
            \foreach \i in {-3,...,3}{
                \draw (v\i\j) -- (v\i\the\numexpr\j+1\relax);
            }
            \foreach \i in {-2,...,3}{ 
                \draw (v\i\j) -- (v\the\numexpr\i-1\relax\the\numexpr\j+1\relax);
            }
        }
        
        \foreach \j in {-2,0,2}{
            \foreach \i in {-3,...,3}{
                \draw (v\i\j) -- (v\i\the\numexpr\j+1\relax);
            }
            \foreach \i in {-3,...,2}{ 
                \draw (v\i\j) -- (v\the\numexpr\i+1\relax\the\numexpr\j+1\relax);
            }
        }


	\end{pgfonlayer}
\end{tikzpicture}}            
            \caption{}
        \end{subfigure}
        \caption{Reference figures for Case 3 in the proof of Theorem~\ref{thm ldb}(d), where gray vertices are codewords.}
        \label{fig:4ldbcases}
    \end{figure}

        \noindent \textbf{Case 3: $\mathbf{|I(v_{0,0})| = 4}$.}  Since three non-codewords cannot be in a straight line
        due to Claim 2(b), we need to consider two valid subcases here based on the adjacency of the two non-codeword neighbors of $v_{0,0}$. 

\medskip
        
        \noindent \textbf{Subcase 1: the two non-codeword neighbors of $\mathbf{v_{0,0}}$ are non-adjacent.} In this subcase, depicted in  Figure~\ref{fig:4ldbcases}(a), $v_{0,0}$ has exactly two non-codeword neighbors: $v_{1, 0}$ and $v_{0, \overline{1}}$. Here we are trying to capture the case where the two non-codeword neighbors of $v_{0,0}$ are non-adjacent (and do not form a straight line together with $v_{0,0}$). Due to vertex transitivity of 
        $\mathbf{G_6}$, one of the non-codeword neighbors (for us, $v_{1,0}$) can be fixed without loss of generality. The choice of the second non-codeword neighbor can be either $v_{0, \overline{1}}$ or 
        $v_{\overline{1}, 1}$. Due to the reflectional symmetry along the $X$-axis, we can choose the former option without loss of generality. 
        
         Note that, 
         $v_{\overline{2},1}, v_{\overline{2},2}$ and $v_{\overline{1},2}$ are codewords due to Claim 2(c). 
         Also, $v_{2,0}$ and $v_{0,\overline{2}}$ must be codewords due to Claim 2(b). 
         Notice that if $v_{2,\overline{1}}$ is a non-codeword, then $v_{1,1}$, $v_{0,2}$, $v_{1,2}$ and $v_{2,1}$ are codewords due to Claim 2(c). 
         On the one hand, if $v_{1,\overline{2}}$ is also a non-codeword, then $v_{\overline{1},\overline{1}}, v_{\overline{2},\overline{1}}, v_{\overline{2},0}$ and $v_{\overline{1},\overline{2}}$ must also be codewords due to similar arguments.
         In this scenario, $ch_1(v_{0,0}) \geq \frac{1}{6} + \frac{1}{2} + 1 + \frac{1}{2} = 2 + \frac{1}{6}$.
         On the other hand, if exactly one of $v_{2,\overline{1}}$ and $v_{1,\overline{2}}$ is a non-codeword, suppose without loss of generality that $v_{1,\overline{2}}$ is a codeword, then, one among $v_{\overline{1},\overline{1}}$ and $v_{2,\overline{2}}$ is a codeword since, $I(\{ v_{0,\overline{1}}, v_{\overline{1},\overline{1}}, v_{2,\overline{2}} \}) = I(\{ v_{\overline{1},\overline{1}}, v_{2,\overline{2}} \})$. 
         If $v_{\overline{1},\overline{1}}$ is a codeword, then also $v_{\overline{2},0}$ (or $v_{2,\overline{2}}$) is a codeword since we have $I(\{v_{\overline{2},0},v_{2,\overline{2}},v_{0,\overline{1}}\})=I(\{v_{\overline{2},0},v_{2,\overline{2}}\})$
         In this scenario, $ch_1(v_{0,0}) \geq \frac{1}{2} + 1 + \frac{1}{3} + \frac{1}{5} = 2 + \frac{1}{30}$ (if $v_{\overline{1},\overline{1}}$ is a codeword) or $ch_1(v_{0,0}) \geq \frac{1}{2} + 1 + \frac{1}{4} + \frac{1}{4} = 2$ (if $v_{2,\overline{2}}$ is a codeword). 
         This implies $ch_2(v_{0,0}) \geq 2$ if $v_{2,\overline{1}}$ is a non-codeword.
         Similarly, one can show that  $ch_2(v_{0,0}) \geq 2$ if $v_{1,\overline{2}}$ is a non-codeword.
         Thus, henceforth in this subcase we can assume that both $v_{2,\overline{1}}$ and $v_{1,\overline{2}}$ are codewords. 

         

Let $v_{2,\overline{2}}$ be a non-codeword. Hence, both $v_{\overline{1},\overline{1}}$ and $v_{1,1}$ are codewords. 
Observe that, if $v_{0,2}$ or $v_{\overline{2},0}$ is a codeword, then $ch_1(v_{0,0}) \geq 2$.
Hence, we can suppose that both $v_{0,2}$ and $v_{\overline{2},0}$ are non-codewords.
Note that, $I(\{  v_{1,0}, v_{0,2}, v_{2,\overline{2}}, u \}) = I(\{ v_{0,2}, v_{2,\overline{2}}, u \})$ for all $u\in \{ v_{2,1}, v_{3,0}, v_{3,\overline{1}} \}$.
Therefore, $  v_{2,1}, v_{3,0}$ and $v_{3,\overline{1}}$ must be codewords.
This implies, 
$ch_1(v_{0,0}) = \frac{1}{3} + 1 + \frac{1}{3} + \frac{1}{4} = 2 - \frac{1}{12}$ and
$ch_1(v_{1,0}) = \frac{1}{3} + \frac{1}{3} + 1 + \frac{1}{3} + \frac{1}{4} = 2 + \frac{1}{4}
$.
Hence, $v_{1,0}$ donates a charge of $\frac{1}{4}$ to $v_{0,0}$, according to (R2).  Therefore, $ch_2(v_{0,0}) \geq ch_1(v_{0,0}) + \frac{1}{4} \geq 2$.

Let $v_{2,\overline{2}}$ be a codeword. 
Notice that, $v_{\overline{1},1}$ and $v_{1,\overline{1}}$ donate $1$ and $\frac{1}{3}$ to $v_{0,0}$ according to (R1).
Next we want to analyze how much charge $v_{0,1}$ and $v_{1,0}$ collectively donate to $v_{0,0}$ according to (R1) and (R2).
Notice that, if one among $v_{0,2}$ and $v_{1,1}$ is a codeword, then $v_{0,1}$ donates a charge of at least $\frac{1}{3}$ to $v_{0,0}$ according to (R1).
However, if both $v_{0,2}$ and $v_{1,1}$ are non-codewords, then the following vertices are all codewords:
$v_{1,2}$ (by Claim 2(b)), $v_{2,2}$, $v_{3,0}, v_{3,\overline{1}}$, $v_{2,1}$, $v_{3,1}$, $v_{3,\overline{2}}$ (by Claim 2(c)).  
This implies, 
$ch_1(v_{1,0}) \geq \frac{1}{4} + \frac{1}{2} + 1 + \frac{1}{3} = 2 + \frac{1}{12}$ 
and
$ch_1(v_{1,1}) \geq \frac{1}{4} + \frac{1}{4} + 1 + \frac{1}{2} = 2$. Thus, $v_{1,0}$ donates a charge of at least $\frac{1}{12}$ to $v_{0,0}$ if $ch_1(v_{0,0}) < 2$, according to (R2). That means, 
$v_{0,1}$ and $v_{1,0}$ collectively donate at least $\frac{1}{3}$ to $v_{0,0}$ according to (R1) and (R2), unless $ch_1(v_{0,0}) \geq 2$. Similarly, due to symmetry, 
$v_{\overline{1},0}$ and $v_{0,\overline{1}}$ collectively donate at least $\frac{1}{3}$ to $v_{0,0}$ according to (R1) and (R2), unless $ch_1(v_{0,0}) \geq 2$.
Therefore, either $ch_1(v_{0,0}) \geq 2$, or  $ch_2(v_{0,0}) \geq 1 + \frac{1}{3}+ \frac{1}{3} + \frac{1}{3} = 2$, which implies $ch_2(v_{0,0}) \geq 2$.

\medskip 

        \noindent \textbf{Subcase 2: the two non-codeword neighbors of $\mathbf{v_{0,0}}$ are adjacent.}
        In this subcase, depicted in  Figure~\ref{fig:4ldbcases}(b), $v_{0,0}$ has exactly two non-codeword neighbors: $v_{0, 1}$ and $v_{1, 0}$.         
        Here we are trying to capture the case where the two non-codeword neighbors of $v_{0,0}$ are adjacent. 
        Since $\mathbf{G_6}$ is also triangle-transitive, 
        we can choose any 
        two adjacent neighbors of $v_{0,0}$ 
        without loss of generality.

        Notice that 
        $v_{\overline{1},2}, v_{0,2},v_{2,0},v_{2,\overline{1}}$ are all codewords to avoid three non-codewords in straight line due to Claim 2(b). 
        Also $v_{\overline{1},\overline{1}}$ must be a codeword as $I(\{ v_{0,0}, v_{0,1}, v_{1,0}, v_{\overline{1},\overline{1}} \}) = I(\{ v_{0,1}, v_{1,0}, v_{\overline{1},\overline{1}} \})$.
        Furthermore, one among $v_{\overline{2}, 1}$ and $v_{1, \overline{2}}$ must be a codeword as  $I(\{ v_{0,0}, v_{\overline{2}, 1}, v_{1, \overline{2}} \}) = I(\{ v_{\overline{2}, 1}, v_{1, \overline{2}} \})$. 
       Let $L$ denote the line through the points $(1,1)$ and $(0,0)$ in $\mathbb{R}^2$. Due to the reflectional symmetry along $L$ of the embedding of $\mathbf{G_6}$ in $\mathbb{R}^2$, without loss of generality we may assume that $v_{\overline{2}, 1}$ is a codeword. 

        Suppose $v_{1,1}$ is a non-codeword. 
        Then $v_{\overline{2},2},v_{\overline{2},3}$ and $v_{\overline{1},3}$  must be codewords as $v_{0,0}, v_{0,1}$ and $v_{1,1}$ are in a curve, due to Claim 2(c).
        By similar arguments, 
        $v_{2, \overline{2}}, v_{3,\overline{1}}$, and
        $v_{3,\overline{2}}$  must also be codewords.
      If $v_{1,\overline{2}}$ is a non-codeword, then $v_{\overline{2},0}$ must be a codeword as $I(\{ v_{0,0}, v_{0,1}, v_{1,\overline{2}}, v_{\overline{2},0} \}) = I(\{  v_{0,1}, v_{1,\overline{2}}, v_{\overline{2},0} \})$.
        In this scenario, $ch_1(v_{{0,0}}) \geq \frac{1}{2} + 1 + \frac{1}{3} + \frac{1}{3} > 2$.
        If $v_{1,\overline{2}}$ is a codeword, then $ch_1(v_{0,0}) \geq \frac{1}{2} + \frac{1}{2} + \frac{1}{2}+ \frac{1}{2} \geq 2$.
        That means, if $v_{1,1}$ is a non-codeword, then $ch_2(v_{0,0}) \geq 2$. 

        Therefore, let $v_{1,1}$ be a codeword.
        Observe that $v_{\overline{1},3}$ must be a codeword as $I(\{ v_{0,1}, v_{0,0}, v_{1,0}, v_{\overline{1},3} \}) = I(\{ v_{0,0}, v_{1,0}, v_{\overline{1},3} \})$.
        Similarly, $v_{3,\overline{1}}$ must also be a codeword. 
        If $v_{\overline{2},0}$ is a codeword, then 
        $ch_1(v_{0,0}) \geq 2$ as one among 
        $v_{0,\overline{2}}, v_{1,\overline{2}}$ and 
        $v_{2,\overline{2}}$ must be a codeword due to Claim 2(b). 
        If $v_{\overline{2},0}$ is a non-codeword,
        then $v_{1, \overline{2}}$ is a codeword as $I(\{ v_{0,0}, v_{0,1}, v_{\overline{2},0}, v_{1,\overline{2}} \}) = I(\{ v_{0,1}, v_{\overline{2},0}, v_{1,\overline{2}} \})$. 
        In this scenario, $ch_1(v_{0,0}) \geq 2$ if  $v_{0,\overline{2}}$ 
        is a codeword. Thus, we may assume that $v_{0,\overline{2}}$ is a non-codeword. 
Also, if both $v_{2,\overline{2}}$ and $v_{\overline{2},2}$  are codewords, then 
        $ch_1(v_{0,0}) \geq \frac{1}{2} + \frac{1}{2} +\frac{1}{2} +\frac{1}{2} =2$. Hence, we may further assume that at least one among $v_{2,\overline{2}}$ and $v_{\overline{2},2}$ 
        is a non-codeword. Without loss of generality, we may assume that $v_{2,\overline{2}}$ is a non-codeword. This implies, $v_{2,1}$ and $v_{3,0}$ are codewords as $I(\{ v_{1,0}, v_{0,1}, v_{2,\overline{2}},u \}) = I(\{ v_{0,1}, v_{2,\overline{2}},u \})$ for $u \in \{v_{2,1}, v_{3,0}\}$. 
        Furthermore, notice that $v_{\overline{2},2}$
        and $v_{1,2}$ cannot be both non-codewords since $I(\{v_{0,1}, v_{\overline{2},2}, v_{1,2}\}) = I(\{v_{\overline{2},2}, v_{1,2}\})$. Based on this, let us consider a few subcases. 
        \begin{enumerate}[(a)]
            \item If both $v_{\overline{2},2}$
        and $v_{1,2}$ are codewords, then $ch_1(v_{0,0})  \geq \frac{1}{2} + \frac{1}{2} + \frac{1}{2} + \frac{1}{3} = 2 - \frac{1}{6}$, 
        $ch_1(v_{0,1})  \geq \frac{1}{2} + \frac{1}{2} + \frac{1}{2} + \frac{1}{2} = 2$, and 
        $ch_1(v_{1,0})  \geq \frac{1}{2} + 1 + \frac{1}{3} + \frac{1}{3} = 2 + \frac{1}{6}$. 
        Thus, $v_{1,0}$ donates an amount of $\frac{1}{6}$ charge to $v_{0,0}$ according to (R2). This implies, 
        $ch_2(v_{0,0}) \geq ch_1(v_{0,0}) + \frac{1}{6} \geq 2$.

         \item If  $v_{\overline{2},2}$ is a codeword, and $v_{1,2}$ is a non-codeword, then 
         $v_{\overline{2},3}$ is a codeword as 
         $I(\{v_{0,1}, v_{0,0}, v_{1,2}, v_{\overline{2},3}\})$ $= I(\{ v_{0,0}, v_{1,2}, v_{\overline{2},3}\})$. 
         Thus, $ch_1(v_{0,0})  \geq \frac{1}{2} + \frac{1}{2} + \frac{1}{2} + \frac{1}{3} = 2 - \frac{1}{6}$, 
        $ch_1(v_{0,1})  \geq \frac{1}{2} + 1 + \frac{1}{3} + \frac{1}{3} = 2 + \frac{1}{6}$, and 
        $ch_1(v_{1,0})  \geq \frac{1}{3} + 1 + \frac{1}{3} + \frac{1}{3} = 2$. 
        Thus, $v_{0,1}$ donates an amount of $\frac{1}{6}$ charge to $v_{0,0}$ according to (R2). This implies, 
        $ch_2(v_{0,0}) \geq ch_1(v_{0,0}) + \frac{1}{6} \geq 2$.

         \item If  $v_{\overline{2},2}$ is a non-codeword, and $v_{1,2}$ is a codeword, then 
         $v_{0,3}$ is a codeword as 
         $I(\{v_{0,1}, v_{1,0}, v_{\overline{2},2}, v_{0,3}\})$ $= I(\{ v_{1,0}, v_{\overline{2},2}, v_{0,3}\})$. 
         Thus, $ch_1(v_{0,0})  \geq \frac{1}{3} + \frac{1}{2} + \frac{1}{2} + \frac{1}{3} = 2 - \frac{1}{3}$, 
        $ch_1(v_{0,1})  \geq \frac{1}{2} + 1 + \frac{1}{3} + \frac{1}{3} = 2 + \frac{1}{6}$, and 
        $ch_1(v_{1,0})  \geq \frac{1}{2} + 1 + \frac{1}{3} + \frac{1}{3} = 2 + \frac{1}{6}$. 
        Thus, each of $v_{0,1},v_{1,0}$ donates an amount of $\frac{1}{6}$ charge to $v_{0,0}$ according to (R2). This implies, 
        $ch_2(v_{0,0}) \geq ch_1(v_{0,0}) + \frac{1}{6} + \frac{1}{6} \geq 2$.
        \end{enumerate}

    \medskip

    \noindent \textbf{Case 4: $\mathbf{|I(v_{0,0})| = 3}$.} 
Since three non-codewords cannot be in a straight line
        due to Claim 2(b), without loss of generality there may be two choices of codeword neighbors of $v_{0,0}$. 

        The first choice is depicted in Figure~\ref{3neighbors4ldb}(a)  
        where $v_{0,0}$ has exactly three codeword neighbors $v_{1,0}, v_{\overline{1},1}, v_{0,\overline{1}} $ and exactly three non-codeword neighbors $v_{0,1},v_{\overline{1},0}, v_{1,\overline{1}}$. 
        This
        is not possible as $I(\{ v_{0,0},v_{0,1},v_{\overline{1},0} \}) = I(\{ v_{0,1},v_{\overline{1},0} \})$.

        \begin{figure}[h!]
        \centering
        \scalebox{.8}{\input{5Neighbours5N-3}}
        \caption{Reference figures for Case 3 in the proof of Theorem~\ref{thm ldb}(d), where gray vertices are codewords.}
        \label{3neighbors4ldb}
    \end{figure}

         The second choice is depicted in Figure~\ref{3neighbors4ldb}(b) where $v_{0,0}$ has exactly three codeword neighbors $v_{1,0}, v_{1,\overline{1}}, v_{0,\overline{1}}$ and
         three non-codeword neighbors $v_{\overline{1},0}, v_{\overline{1},1}$ and $v_{0,1}$.
         Observe that $v_{\overline{1}, \overline{1}}, v_{\overline{1},2}, v_{\overline{2},1}, v_{1,1}, v_{\overline{2},2}, v_{\overline{2},0},v_{0,2}$ are codewords since there cannot be three consecutive non-codewords in a straight line, due to Claim 2(b). 
        Moreover, $v_{\overline{3},2}, v_{\overline{2},3}$ and $v_{\overline{3},3}$(resp., $ v_{1,\overline{2}}, v_{2,\overline{1}}$ and $v_{2,\overline{2}}$) must be codewords as $v_{\overline{1},0}, v_{\overline{1},1}, v_{0,1}$ (resp., $v_{\overline{1},0}, v_{0,0}, v_{0,1}$) are in a curve, due to Claim 2(c).
        Similarly, $v_{\overline{3},3}$ must be a codeword. Also, $v_{\overline{2}, \overline{1}}$ is a codeword as 
        $I(\{ v_{\overline{1},0},v_{\overline{1},1}, v_{0,0},v_{\overline{2},\overline{1}} \}) = I(\{ v_{\overline{1},1}, v_{0,0},v_{\overline{2},\overline{1}} \})$.
        Similarly, $v_{1,2}$ is a codeword.

        Suppose both $v_{2,0}$ and $v_{0,\overline{2}}$ are codewords. 
        Then $ch_1(v_{0,0}) \geq 2$, which in turn implies $ch_2(v_{0,0}) \geq 2$. 
Therefore, suppose that both $v_{2,0}$ and $v_{0,\overline{2}}$ are not codewords.
Without loss of generality, suppose that
 $v_{2,0}$ is a non-codeword. 
 Then both $v_{\overline{1}, 3}$ and $v_{0,3}$ are codewords 
as
$I(\{v_{0,1}, v_{2, 0}, v_{\overline{1}, 1}, u\}) = I(\{v_{2, 0}, v_{\overline{1}, 1}, u\})$ for all 
$u \in \{v_{\overline{1}, 3}, v_{0,3}\}$. 

If $v_{0, \overline{2}}$ is also a non-codeword, then $v_{\overline{3}, 1}$ and $v_{\overline{3}, 0}$ 
are codewords, similarly. 
In this scenario, 
$ch_1(v_{0,0}) \geq \frac{1}{3} + 1 + \frac{1}{3} = 2 - \frac{1}{3}$,
$ch_1(v_{\overline{1}, 1}) \geq \frac{1}{2} + 1 + \frac{1}{2} = 2$, 
$ch_1(v_{0,1}) \geq \frac{1}{2} + 1 + \frac{1}{3} + \frac{1}{3} = 2 + \frac{1}{6}$, and 
$ch_1(v_{\overline{1},0}) \geq \frac{1}{2} + 1 + \frac{1}{3} + \frac{1}{3} = 2 + \frac{1}{6}$. 
Thus, each of $v_{0,1}$ and $v_{\overline{1}, 0}$ donates at least $\frac{1}{6}$ to $v_{0,0}$ according to (R2). Therefore, $ch_2(v_{0,0}) \geq ch_1(v_{0,0}) + \frac{1}{6} +\frac{1}{6} \geq 2$.

Hence, let us suppose that $v_{0, \overline{2}}$ is a codeword. Moreover, if $v_{\overline{3}, 1}$ is also a codeword, then 
$ch_1(v_{0,0}) \geq \frac{1}{3} + 1 + \frac{1}{2} = 2 - \frac{1}{6}$,
$ch_1(v_{\overline{1}, 1}) \geq \frac{1}{2} + 1 + \frac{1}{2} = 2$, 
$ch_1(v_{0,1}) \geq \frac{1}{2} + 1 + \frac{1}{3} + \frac{1}{3} = 2 + \frac{1}{6}$, and 
$ch_1(v_{\overline{1},0}) \geq \frac{1}{2} + \frac{1}{2} + \frac{1}{2} + \frac{1}{2} = 2$.
Thus,  $v_{0,1}$  donates at least $\frac{1}{6}$ to $v_{0,0}$ according to (R2). Therefore, $ch_2(v_{0,0}) \geq ch_1(v_{0,0}) +\frac{1}{6} \geq 2$.

On the other hand, if $v_{\overline{3}, 1}$ is a non-codeword, then $v_{\overline{1}, \overline{2}}$ is a codeword as $I(\{ v_{\overline{1},0}, v_{\overline{1},\overline{2}}, v_{\overline{3},1}, v_{0,0} \}) =I(\{  v_{\overline{1},\overline{2}}, v_{\overline{3},1}, v_{0,0} \})$. 
In this scenario, 
$ch_1(v_{0,0}) \geq \frac{1}{3} + 1 + \frac{1}{2} = 2 - \frac{1}{6}$,
$ch_1(v_{\overline{1}, 1}) \geq \frac{1}{2} + 1 + \frac{1}{3} = 2- \frac{1}{6}$, 
$ch_1(v_{0,1}) \geq \frac{1}{2} + 1 + \frac{1}{3} + \frac{1}{3} = 2 + \frac{1}{6}$, and 
$ch_1(v_{\overline{1},0}) \geq \frac{1}{2} + 1 + \frac{1}{3} + \frac{1}{3} = 2 + \frac{1}{6}$.
Thus,  each of $v_{0,1}$ and $v_{\overline{1},0}$  donates at least $\frac{1}{12}$ to both $v_{0,0}$
and $v_{\overline{1}, 1}$ according to (R2). Therefore, $ch_2(v_{0,0}) \geq ch_1(v_{0,0}) +\frac{1}{12} + \frac{1}{12} \geq 2$.

Therefore, $ch_2(v_{0,0}) \geq 2$ when $|I(v_{0,0})| = 3$ concluding the proof.

  \medskip

        \noindent\textbf{Case 5: $\mathbf{|I(v_{0,0})| \leq 2}$.}
        Observe that, regardless of how the codeword 
        neighbors (at most $2$) of $v_{0,0}$ are positioned, there 
        will exist three non-codeword vertices on a straight line, contradicting Claim 2(b). Thus, this particular case never occurs. \qed

\section{Conclusion}\label{SecConclusion}

In this paper, we have given exact minimum densities $D^A_{1,2}=D^B_{1,2}=\frac{1}{3}$, $D^A_{1,\ell}=1$ for each $\ell\geq3$, $D^B_{1,3}=\frac{4}{7}$ and $D^B_{1,\ell} = \frac{2}{3}$, for all $\ell \geq 4$ in the infinite triangular grid. Together with known results for usual location-domination, these results cover every case for $(1,\leq \ell)$-LDA and -LDB codes. By~\cite{lehtilathree}, the result $D^B_{1,7} = \frac{2}{3}$ gives the maximum density of open irredundant sets as well as the minimum density of MILD codes in the infinite triangular grids. 

We have presented constructions in each case which obtain the shown minimum densities. Our lower bound proofs are based on the discharging method, as is typical for these types of results. However, interestingly we have managed to implement two generic discharging rules which alone have yielded the exact lower bounds in all three cases. This is atypical since often different lower bound proofs based on the discharging method require varying numbers of different rules.

Additionally, software implementations~\cite{griddy-paper,griddy-repo,adage-repo,adage-paper} exist to find the bounds using discharging and assignments for similar problems. 
To the best of our knowledge, no computer-aided approach has previously been applied to the $(r,\leq \ell)$-LDA or -LDB codes considered here. 
Moreover, none of the proofs in this article are computer-aided.

While we have focused on only the case with $r=1$, the cases with larger $r$ are left for future research. Another interesting related open problem is obtaining the exact minimum density of $(1,\leq2)$-LDB codes in the infinite square grid. The current best lower bound is $\frac{2}{5}$ while the best known construction has density $\frac{5}{11}$. Yet another future research direction is considering LDA/LDB codes in the infinite hexagonal grid. There are currently no articles on the topic beyond the works of usual location-domination~\cite{honkala2006locating} (giving the density for $\ell=1$) and works considering $(1,\leq\ell)$-identifying codes which provide upper bounds for the minimum density of LDA/LDB codes~\cite{honk04c}.

\subsubsection*{Acknowledgments.} The research of Tuomo Lehtilä was in part supported by the Finnish Cultural Foundation and Academy of Finland grants 338797 and 358718.
This work is also partially supported by the IFCAM project “Applications of graph homomorphisms” (MA/IFCAM/18/39).

\subsubsection*{Disclosure of Interests.}
The
authors have no competing interests to declare that are relevant to the content of this article.

    \bibliographystyle{abbrv}
    \bibliography{biblio}

\end{document}